\documentclass[12pt, reqno]{amsart}
\usepackage{amsmath,latexsym,amssymb,amsfonts,amsbsy, amsthm}
\allowdisplaybreaks[4]
\usepackage{graphicx}
\usepackage{empheq}
\usepackage{ulem,cite}
\usepackage{amsthm}
\usepackage[margin=1in]{geometry}
\usepackage{tikz}
\usepackage{latexsym}
 \graphicspath{{Figures/}}
\usepackage{xcolor}
\usepackage{ulem}

\numberwithin{equation}{section}
\newtheorem{Proposition}{Proposition}
\numberwithin{Proposition}{section}
\newtheorem{Theorem}{Theorem}
\numberwithin{Theorem}{section}
\newtheorem{Lemma}{Lemma}
\numberwithin{Lemma}{section}
\newtheorem{Remark}{Remark}
\numberwithin{Remark}{section}
\newtheorem{Corollary}{Corollary}
\numberwithin{Corollary}{section}
\newtheorem{definition}{Definition}

\newcommand{\pl}{\partial}
\newcommand{\lt}{\left}
\newcommand{\rt}{\right}
\newcommand{\ep}{\epsilon}
\newcommand{\epn}{{\epsilon_n}}
\newcommand{\andf}{\quad{\rm and}\quad}
\newcommand{\R}{\mathbb{R}}

\begin{document}
\title[]{Regularity and dynamics of weak solutions for one-dimensional  compressible Navier-Stokes equations with vacuum}

\author{Jin Tan$^\ast$, Yan-Lin Wang and  Lan Zhang}
\address[J. Tan]{Department of Mathematics,The Chinese University of Hong Kong, Shatin, Hong Kong,  China}
\email{jtan@math.cuhk.edu.hk}
\address[Y.-L. Wang]{School of Mathematical Sciences,  Zhejiang University of Technology,  Hangzhou 310023, China}
\email{ylwang@zjut.edu.cn}
\address[L. Zhang]{ School of Mathematics and Statistics, Wuhan University of Technology, Wuhan 430070, China}
\email{zhanglan91@whut.edu.cn}

\date{\today}
\thanks{$^\ast$corresponding author}

\let\origmaketitle\maketitle
\def\maketitle{
  \begingroup
  \def\uppercasenonmath##1{} 
  \let\MakeUppercase\relax 
  \origmaketitle
  \endgroup
}
\maketitle

\begin{abstract}
In the spirit of  D. Hoff's weak solution theory for the compressible Navier-Stokes equations (CNS) with bounded density,   in this paper we establish the  global  existence and regularity properties of finite-energy weak solutions to an  initial boundary value problem  of  one-dimensional CNS with general  initial data and vacuum.   The core of our proof is a global in time a priori estimate  for one-dimensional CNS that holds for any $H^1$  initial velocity and bounded initial density not necessarily strictly positive: it  could be  a density patch  or a vacuum bubble.   We also establish that the velocity and density decay exponentially to equilibrium.  As a by-product,   we obtain the quantitative dynamics of aforementioned two vacuum states.
\end{abstract}

\section{Introuduction}

In this paper,   we are concerned with the following initial boundary value problem (IBVP)  of  one-dimensional compressible Navier-Stokes (CNS) equations:
\begin{subequations}\label{equ}
\begin{align}
&\pl_t \rho+\pl_x(\rho u)=0, &&\ \ {\rm in}\ \ \Omega, \label{equ-a}\\
&\pl_t(\rho u)+\pl_x (\rho u^2)+\pl_x p =\pl_x(\mu(\rho)\pl_xu), &&\ \ {\rm in}\ \  \Omega, \label{equ-b}\\
&(\rho, u)(0, x)=(\rho_0, u_0)(x), &&\ \ {\rm in}\ \  \Omega,\label{equ-c}\\
& u=0,\ \ \ p(\rho)-\mu(\rho)\pl_x u=0, &&\ \ {\rm on}\ \ \pl  \Omega,\label{equ-d}
\end{align}
\end{subequations}
where $ \Omega$ is taken as $ \Omega:=(0, 1)$ and  $(t, x)\in   \mathbb{R}_+\times  \Omega.$  Here,  the unknowns $\rho$ and  $u$   denote fluid density  and velocity,  respectively.    The pressure  $p=p(\rho)$ is a given function of density.
The viscosity coefficient $\mu(\rho)$ is  not necessarily  a positive constant,   as it is well known that the viscosity of a gas depends on the temperature   and thus on the density in the current isentropic case.    To cover a wide range of    gas dynamics,   in the sequel of the paper  the pressure law and viscosity coefficient will be given  in the general way that:
\begin{align}
& p(s)\in C^1(\mathbb{R}_+),  ~~ p(s), ~p'(s)\geq 0\andf p~~{\rm is~ convex};\label{pcon}\\
&\mu(s)\in C^1(\mathbb{R}_+),  ~~   \mu_*\leq \mu(s)\leq \mu^*~{\rm for~some~constants}~\mu_*, ~\mu^*>0,\notag\\
&\quad\qquad\quad\qquad \andf  s\,\int_1^s \frac{\mu(\bar s)}{\bar s} \,d\bar s~\quad  {\rm is~ convex}.\label{mucon}
\end{align}

Examples of admissible pressure and viscosity include
$$p(\rho)=C_1 \rho^\gamma\andf  \mu(\rho)=\mu_*+C_2\rho^\beta,   \quad{\rm for~constants} ~C_1>0, \,C_2\geq0, ~~ \gamma\geq 1,\,\beta\geq 0,$$
which corresponds to polytropic perfect gases.

Thanks to the boundary condition \eqref{equ-d},  it is clear that  the total mass and momentum of sufficiently smooth solutions of IBVP \eqref{equ} are conserved through  evolution, namely,  for any  $t>0,$
\begin{equation}\label{Laws}
\int_{ \Omega} \rho(t, x) \, dx=\int_{ \Omega} \rho_0  \, dx=M_0 \andf  \int_ \Omega (\rho u)(t, x) \, dx =\int_ \Omega (\rho_0 u_0)(x) \, dx.
\end{equation}
We shall always assume that
\begin{equation}\label{initialdata-Cond1}
\int_{ \Omega} \rho_0  \, dx=1\andf \int_ \Omega (\rho_0 u_0)(x) \, dx=0.
\end{equation}
Let us denote by $e$  the potential energy of the fluid.   By the relation $\rho e''(\rho)=p'(\rho)$ and the  normalization  $e(1)=e'(1)=0,$  then   $e(\rho)$ is given by
\begin{align*}
e(\rho)=\rho \int_1^{\rho} \frac{p(s)}{s^2} ds- p(1)(\rho -1).
\end{align*}
Hence,  $e$ is a strictly convex function and $\|e\|_{L^1(\Omega)}$ is essentially equivalent to $\|\rho-1\|_{L^2(\Omega)}^2.$ Next,   introduce  the total energy
\begin{equation*}
E(t):=\int_\Omega [ 2 e(\rho)+\rho u^2](t, x) \,dx,
\end{equation*}
one has the following energy balance:  for any $t>0,$
\begin{equation}\label{balance}
E(t)+2\int_0^t\int_\Omega (\mu(\rho) |\pl_x u|^2)(\tau, x) \,dxd\tau=E_0:=E(0).
\end{equation}

 \bigbreak

The compressible Navier-Stokes equations have been extensively investigated since the foundational work of Nash \cite{MR149094}, in which the local existence and uniqueness of classical solutions for isentropic flows in the absence of vacuum (where density is zero) is  established.   Building upon this fundamental local well-posedness result, substantial progress has been achieved in constructing \textit{global-in-time solutions}, encompassing both vacuum and non-vacuum cases, through the development of three distinct frameworks: the strong solution theory, Lions' weak solution theory, and Hoff's weak solution theory with bounded density and small energy, as detailed in the following contributions.

\underline{Strong solution theory.} The global existence of classic or strong solutions with arbitrary large initial data is not known generally. Only the one-dimensional theory is quite satisfactory.
The foundational results on global well-posedness for one-dimensional  flows with large initial data away from vacuum   trace back to the seminal contributions of Kanel$^\prime$ \cite{MR227619} and  Kazhikhov-Shelukhin \cite{MR468593}. For higher dimensional case, initial data are usually required close to constant equilibrium and global small smooth solutions  are constructed based on the corresponding linearized system, see for example Matsumura and Nishida \cite{MR564670}, and also Danchin \cite{MR1779621} for a critical functional framework. Surprisingly, Va\u{\i}gant and Kazhikhov \cite{MR1375428} investigated a class of two-dimensional isentropic compressible Navier-Stokes equations with specified density-dependent viscosity coefficients and provided the existence and uniqueness of global strong solutions with general large initial data.
      Regarding strong solutions with possible vacuum, the problem of local well-posedness for the isentropic case was resolved by Salvi and Stra\u{s}kraba \cite{MR1217657}, followed by contributions from Cho, Choe, and Kim \cite{MR2038120}, and Cho and Kim \cite{MR2223483}. The first proof of global strong solutions for the isentropic compressible Navier-Stokes equations allowing large initial oscillations and   vacuum  was given by Huang, Li, and Xin \cite{MR2877344}. Further developments in this direction were contributed by Li and Xin \cite{MR3923730,MR3597161}.   Improvements in Va\u{\i}gant and Kazhikhov type global large solutions to include vacuum were established in  Jiu-Wang-Xin \cite{MR3247365}.  See also Huang and Li \cite{MR3505779} for more general parameter case.  

 \underline{ Lions' weak solution theory.} Consider the isentropic compressible Navier-Stokes equations with general initial data, Lions \cite{MR1637634} first proved the global existence of finite-energy weak solutions for polytropic perfect gases with $\gamma>\frac{9}{5}$. Subsequent work by Feireisl, Novotn\'{y}, and Petzeltov\'{a} \cite{MR1867887} relaxed this requirement to $\gamma>\frac32$, while Jiang and Zhang \cite{MR2005201} further generalized it to $\gamma>1$ in the axisymmetric case. Later, Feireisl and Novotn\'{y} \cite{MR2499296} generalized the analysis to the case of compressible viscous heat-conducting fluids. For more recent advances in the density-dependent viscosity case, see  Bresch and  Jabin \cite{MR3862947} and references therein.
 Solutions in this very general class possess rather little regularity so that analysis of their detailed qualitative properties seems difficult. In particular, the question of  uniqueness is largely open.

\underline{Hoff's weak solution theory.} It is started from the work \cite{MR896014} of Hoff, in which the global existence of  solutions for one-dimensional isentropic flows with large initial velocity  in $L^2$ and   initial density  in $L^2\cap BV$ is established. Chen, Hoff, and Trivisa \cite{MR1789926} later generalized this result to the non-isentropic case in a bounded domain. See also recent developments in \cite{MR4373171,MR4659290,MR4444077,MR4579720,MR4695791}.
For multi-dimensional case, this type of solutions is limited to have small initial energy and has been constructed by Hoff  \cite{MR1088275,MR1142276}. We also mention its extension to the compressible MHD system in \cite{MR2927617}.
Solutions in this class usually have bounded density, which is physically relevant.  Linearization method is not valid in this class, and solutions may exhibit truly nonlinear and physically interesting effects. For example, it allows piecewise smooth density profiles. Moreover, these solutions have   more regularity properties than Lions' type weak solutions such that a reasonable uniqueness theory is possible, at least for one-dimensional case,   see Danchin, Fanelli, and Paicu \cite{MR4047647}.
For higher dimensional case, the uniqueness issue of Hoff type weak solutions presents significant challenges. Recently, Danchin and Mucha \cite{MR4642822} made important progress in this direction. More precisely, under the assumptions that the initial velocity lies in $H^1(\mathbb{T}^3)$ and initial density is merely bounded and nonnegative, and that bulk viscosity is sufficiently large, they established the global existence  of solutions in the two-dimenisonal case. Moreover, in the case that the  pressure law is linear, the solution is unique.
Similar result concerning existence and uniqueness were also obtained in $\mathbb{T}^3$, either locally in time for general large data or globally under an appropriate scaling-invariant smallness condition on the velocity field (with no smallness restriction on the density). Exponential decay of the corresponding solutions in the two-dimensional torus was proved later by Danchin and Wang \cite{MR4718745}.

To the best of our knowledge, Hoff's weak solution theory is not known for the one-dimensional  isentropic flows with arbitrary large initial data and vacuum. This unsatisfactory gap in the theory motivates the present work.

\subsection{Main results}

In this paper, we   establish the  global  existence,   regularity properties and  decay estimates of finite-energy weak solutions \textit{\`{a} la Hoff} for the IBVP \eqref{equ} with   large rough  initial data in a vacuum.    Throughout the paper,  we use   $\dot v$ and sometimes $\frac{D}{Dt} v $ to denote the material derivative of $v,$ i.e.   $\dot{v}=\pl_t v+u\pl_x v.$

To state our result,  let us first give the definition of weak solutions.
\begin{definition}\label{Def1}
A pair of function $(\rho,u)$ is said to be a Hoff type  global weak solution of system \eqref{equ} provided that $\rho\geq 0$ almost everywhere and
\begin{align*}
  &\rho\in L^\infty(\R_+\times \Omega)\cap \mathcal{C}(\R_+; L^r(\Omega)),~ \sqrt{\rho}u \in L^\infty(\mathbb{R}_+, L^2(\Omega)),\\
  &\partial_x u \in L^2(\mathbb{R}_+\times \Omega),
\end{align*}
and that
\begin{align*}
  \left.\int_\Omega\rho\phi \,dx\right|^{t_2}_{t_1}&=\int_{t_1}^{t_2}\int_\Omega(\rho\pl_\tau\phi+\rho u\pl_x\phi)\,dxdt,\\
  \left.\int_\Omega\rho u\phi \,dx\right|^{t_2}_{t_1}&=\int_{t_1}^{t_2}\int_\Omega\{\rho u(\pl_\tau\phi+  u\pl_x\phi)+(p(\rho)-\mu(\rho)\pl_xu)\pl_x\phi\}\,dxdt,
\end{align*}
for any $t_1$ and $t_2$ satisfying $0<t_1<t_2< \infty$ and test function $\phi\in \mathcal{C}_0^{\infty}(\mathbb{R}_+\times\Omega)$.
\end{definition}

\begin{Theorem}\label{thm1}
Assume that the  pressure $p(\rho)$ and  the viscosity $\mu(\rho)$  satisfy  conditions  \eqref{pcon}--\eqref{mucon}.   Consider any  initial data $(\rho_0, u_0)$ satisfy \eqref{initialdata-Cond1},  boundary condition    $u_0|_{\pl\Omega}=0$ and
\begin{align}\label{initialdata-Cond2}
0\leq \rho_0(x)\leq  \rho^*~~{\rm ~for~ some ~  positive~ constant}~ \rho^*,  \andf  \sqrt{\rho_0}u_0\in L^2( \Omega).
\end{align}
Then,   {\rm IBVP} \eqref{equ} has a  global-in-time weak solution $(\rho, u)$  in the sense of Definition \ref{Def1} that  fulfilling the conservation laws in \eqref{Laws} and
\begin{itemize}
\item Regularity properties:
\begin{align*}
 u(t,  \cdot), ~( \mu(\rho)\partial_x u- p(\rho))(t, \cdot)\in H^1(\Omega), \quad~ (\rho\dot u)(t, \cdot)\in L^2(\Omega)\quad{\rm for}~t>0.
\end{align*}
 In~particular,~ $u(t, \cdot)$~{is~Lipschitz~continuous~instanously};
\item Exponential decay:
\begin{equation}
\|(\rho-1)(t, \cdot)\|_{L^2}+ \| u(t,  \cdot)\|_{L^2}\leq C e^{-\alpha t},  \quad {\rm for}~t>0,
\end{equation}
where   $C$ and $\alpha$ are positive and  depending only on $E_0,  \mu_*,   \mu^*$ and $\rho^*.$
\end{itemize}
Moreover,  if in addition $\partial_x u_0\in L^2(\Omega),$   then $u$ satisfies the additional estimates:
\begin{align*}
  &\partial_x u\in L^{\infty}(0,  T;L^2(\Omega)), \quad \sqrt{t\rho}\dot{u}\in L^{\infty}(0, T;L^2(\Omega)) \\
  &{\rm and}\quad    \mu(\rho)\partial_x u- p(\rho)    , \in L^2(0, T;  H^1(\Omega)), \quad {\rm for ~any}~T>0,
\end{align*}
and furthermore,  this weak solution is unique in the class:
\begin{align*}
  &\rho\in L^{\infty}((0,  T)\times \Omega), \quad \sqrt{\rho} u\in L^{\infty}(0, T;L^2(\Omega)) \\
  &{\rm and}\quad     \partial_x u \in L^1(0, T;  L^\infty(\Omega)),  ~\sqrt{t}\partial_x u\in  L^2(0, T;  L^\infty(\Omega)).
\end{align*}
\end{Theorem}

\medskip

As an application of Theorem \ref{thm1},    we are able to demonstrate quantitatively the dynamics of two typical vacuum states: density patch and vacuum bubble,  which are physically interesting.
Before presenting our corollary,  let us  review some known results for the one-dimensional vacuum free boundary problem.

\underline{Constant viscosity case.} When the viscosity coefficient $\mu$ is constant, \cite{MR1117422} demonstrated the failure of continuous dependence on initial data for Navier-Stokes solutions with vacuum. For the free boundary problem with one fixed boundary and one vacuum interface, \cite{MR981519} established the existence of global weak solution. Similar results for spherically symmetric viscous gases were obtained in \cite{MR1227730}. Further analysis of solution regularity and asymptotic behavior near vacuum interfaces was developed in \cite{MR1766564}.

  \underline{Density-dependent viscosity ($\mu(\rho)=\rho^\lambda$) with continuous vacuum connection.} Under boundary conditions $\rho(t,a(t)) = \rho(t,b(t)) = 0$ (as opposed to the stress-free condition $(p(\rho) - \mu(\rho)\partial_x u)(t,a(t)) = 0$), \cite{MR1929151} proved the local existence of weak solutions for $\frac{1}{2} < \lambda \leq \gamma - \frac{1}{3}$. Global existence for $0<\lambda<\frac{2}{9}$ was first shown in \cite{MR1936794}, and later extended to wider $\lambda$ ranges in \cite{MR2237707, MR2644149} and the references therein.

\underline{Density-dependent viscosity ($\mu(\rho)=\rho^\lambda$)  with  discontinuous vacuum connection.} For discontinuous initial density connecting to vacuum, \cite{MR1485360} and \cite{MR882389} studied the local existence of weak solutions under the stress-free condition $(p(\rho)-\mu(\rho)\partial_x u)(t, a(t)) =0.$ Global existence and uniqueness results appeared in \cite{MR2254008,MR1980822,MR2373221,MR1843291,MR2563807} and the references therein.

Concerning the large time behavior of the weak solution to the free boundary problem of system \eqref{equ-a}-\eqref{equ-b} with constant viscosity, Luo, Xin and Yang \cite{MR1766564} obtained
the following interface behavior:
\begin{equation*}
  c_1\rho_0(x_1)(1+c_2\rho_0^{\gamma}(x_1)t)^{-1/\gamma}\leq \rho(t,x)\leq c_3\rho_0(x_1)(1+c_4\rho_0^{\gamma}(x_1)t)^{-1/\gamma},
\end{equation*}
and
\begin{equation*}
  c_5(1+t)^{1/\gamma}\leq b(t)-a(t)\leq c_6(1+t)^{1/\gamma}.
\end{equation*}
For density dependent viscosity $\mu(\rho)=\rho^{\lambda}$, Zhu \cite{MR2563807} showed that the density function tends to zero in the following way, when the initial density connects to vacuum with discontinuities
\begin{equation*}
  \rho\leq \frac{C}{(1+t)^{\kappa}},
\end{equation*}
and
\begin{equation*}
  b(t)-a(t)\geq C(1+t)^{\kappa},
\end{equation*}
where $\kappa$ depends on $\lambda>0$ and $\gamma>1$, $C$ is a positive constant depending on the initial data but independent of time. For the case that the initial density connects to vacuum with discontinuities, similar algebraic decay estimates can be found in \cite{MR2644149,MR3926039} and the references therein.

In contrast to the aforementioned vacuum free boundary problems in unbounded domain, where the asymptotic behavior of the interface typically exhibits algebraic rates, our analysis provides a quantitative description of the vacuum interface motion with fundamentally different dynamical characteristics. Specifically, by exploiting the exponential decay estimates established in Theorem \ref{thm1}, we are able to demonstrate that the vacuum states in a bounded domain are being compressed exponentially fast.   See Figure \ref{fig:outward} and Figure \ref{fig:inward} below for illustrations.

\begin{Corollary}\label{cor1}
Let $0<a_0<b_0<1.$ Consider any global solution $(\rho, u)$ that obtained from Theorem \ref{thm1}  supplemented with initial density given by
\begin{itemize}
\item A density patch  $\rho_0(x)=\frac{1}{b_0-a_0}{\bf 1}_{\{ a_0<x< b_0\}}(x).$  Then the vaccum boundary points $a(t),~b(t)$ defined  in \eqref{ODE} satisfying
\begin{align}\label{cor-est1}
   |a(t)|\leq Ce^{-\alpha t} \andf |1-b(t)|\leq Ce^{-\alpha t}.
\end{align}
\item A vacuum bubble    $\rho_0(x)=  \frac{1}{1+a_0-b_0} {\bf 1}_{\{0\leq x<a_0\,\cup \,b_0<x\leq1\}}.$   Then the vaccum boundary points $a(t),~b(t)$ defined in  \eqref{ODE} satisfying
\begin{align}\label{cor-est2}
  |a(t)-x^\infty|\leq Ce^{-\alpha t} \andf |b(t)-x^\infty|\leq Ce^{-\alpha t},
\end{align}
with $x^\infty=\int_0^{a_0}\rho_0(x)\,dx.$
\end{itemize}
\end{Corollary}
 \bigbreak

\begin{figure}[htbp]
 \centering
    \begin{minipage}{0.48\textwidth}
        \centering
        \begin{tikzpicture}[scale=0.8]
\draw[->] (-1,0)--(0,0)node[below left]{\tiny 0 } --(2,0) node[below]{\tiny $a(t)$ }--(4,0) node[below]{\tiny $b(t)$ }--(6,0) node[below left]{\tiny 1 }    -- (6.5,0) node[below]{\tiny $ x$};
\draw[->] (0,-1)--(0,1)node[left]{\tiny 1} --(0,3)node[left]{\tiny $\frac{1}{b_0-a_0}$}  --(0,3.5) node[right]{\tiny $\rho_0$};
\draw[->] (2,1.5)--(1.5,1.5);
\draw[->] (4,1.5)--(4.5,1.5);
\draw (0,0) circle [radius=1.5pt];
\draw (0,1) circle [radius=1.5pt];
\draw (6,1) circle [radius=1.5pt];
\draw (2,3) circle [radius=1.5pt];
\draw (4,3) circle [radius=1.5pt];
\draw[dashed] (0, 1) -- (6, 1);
\draw[thin] (6, 4) -- (6,-1);
\draw[dashed] (2, 3) -- (2, 0);
\draw[dashed] (4, 3) -- (4, 0);
\draw[thick] (2, 3)--(4, 3);
\draw[red,dashed] (2, 3) to  (4, 0);
\draw[red,dashed] (2, 1.5) to  (3, 0);
\draw[red,dashed] (3, 3) to  (4, 1.5);
\draw[blue,dashed] (1, 1) to  (0, 0);
\draw[blue,dashed] (2, 1) to  (1, 0);
\draw[blue,dashed] (3, 1) to  (2, 0);
\draw[blue,dashed] (4, 1) to  (3, 0);
\draw[blue,dashed] (5, 1) to  (4, 0);
\draw[blue,dashed] (6, 1) to  (5, 0);
\end{tikzpicture}
\caption{Density patch}
\label{fig:outward}
    \end{minipage}
    \hfill
    \begin{minipage}{0.48\textwidth}
        \centering
        \begin{tikzpicture}[scale=0.8]
\draw[->] (-1,0)--(0,0)node[below left]{\tiny 0 } --(2,0) node[below]{\tiny $a(t)$ }--(3,0) node[below]{\tiny $x^\infty$ }--(4,0) node[below]{\tiny $b(t)$ }--(6,0) node[below left]{\tiny 1 }    -- (6.5,0) node[below]{\tiny $ x$};
\draw[->] (0,-1)--(0,1)node[left]{\tiny 1}--(0,3)node[left]{\tiny $\frac{1}{1+a_0-b_0}$}  --(0,3.5) node[right]{\tiny $\rho_0$};
\draw[->] (2,1.5)--(2.5,1.5);
\draw[->] (4,1.5)--(3.5,1.5);
\draw[fill] (0,3) circle [radius=1.5pt];
\draw[fill] (6,3) circle [radius=1.5pt];
\draw (2,3) circle [radius=1.5pt];
\draw (3,1) circle [radius=1.5pt];
\draw (4,3) circle [radius=1.5pt];
\draw[fill] (0,1) circle [radius=1.5pt];
\draw[fill] (6,1) circle [radius=1.5pt];
\draw[dashed] (0, 1) -- (6, 1);
\draw[thin] (6, 4) -- (6,-1);
\draw[dashed] (2, 3) -- (2, 0);
\draw[dashed] (4, 3) -- (4, 0);
\draw[thick] (0, 3)--(2, 3);
\draw[thick] (4, 3)--(6, 3);
\draw[red,dashed] (0, 3) to  (2, 0);
\draw[red,dashed] (0, 1.5) to  (1, 0);
\draw[red,dashed] (1, 3) to  (2, 1.5);
\draw[red,dashed] (4, 3) to  (6, 0);
\draw[red,dashed] (4, 1.5) to  (5, 0);
\draw[red,dashed] (5, 3) to  (6, 1.5);
\draw[blue,dashed] (1, 1) to  (0, 0);
\draw[blue,dashed] (2, 1) to  (1, 0);
\draw[blue,dashed] (3, 1) to  (2, 0);
\draw[blue,dashed] (4, 1) to  (3, 0);
\draw[blue,dashed] (5, 1) to  (4, 0);
\draw[blue,dashed] (6, 1) to  (5, 0);
\draw[black,dashed, line width=1pt] (3, 1) to  (3, 0);
\end{tikzpicture}
\caption{Vacuum bubble}
\label{fig:inward}
    \end{minipage}
\end{figure}

 Let us make a few remarks on Theorem \ref{thm1}.
\begin{Remark}\ \

 \begin{itemize}
 \item[1.] A major difficulty on the proof of Theorem \ref{thm1} is the lack of regularity of the density, which also highlights our result on the other side.
Moreover,  the case that the initial density contains a vacuum gives an extreme difficulty in obtaining robust compactness property for the density. Note that  Hoff's original Lagrangian type strategy \cite{HOff98} for proving strong convergence of density fails if there is a vacuum, because he  use essentially the condition that density is far away from vacuum (it is satisfied initially) so that the Jacobian of particle trajectories is non-degenerate.  Our strategy is quite different, it replies on  a global-in-time a priori estimate (see Lemma \ref{lemma3} and also Lemma \ref{lemma4})  for \eqref{equ} that valid for any $H^1$  initial velocity and bounded initial density not necessarily strictly positive.
More precisely,  we observe that Hoff type weak solutions are ``strong solutions" instanously  after time evolution, even in the presence of vacuum. With those additional regularity properties of the velocity, we finally obtain the global existence and (conditional) uniqueness  of Hoff type weak solutions.

 \item[2.]  The dependence of viscosity on density makes it hard to trace the trajectory of vacuum state. In fact,    degenerate   viscosity  coefficient may even  lead to singularity formation for \eqref{equ}.  For example,  in  \cite{MR2410901} it demonstrates that any possible vacuum state vanishes in finite time and the velocity blows up in finite time if a vacuum state appears initially.     Our   Theorem \ref{thm1}  excludes formation of singularity  for general Hoff type weak solutions with vacuum  whenever the   density-dependent viscosity is not degenerate.  Actually, our argument is still valid even if
the  viscosity function $\mu(s)$ vanishes as $s$ tends to zero, in the case that initial densities are far away from vacuum.
     Moreover, if $\mu\equiv {\rm Constant},$ exponential time decay of $\partial_x u$ is obtained in Proposition \ref{decay2}.

    \item[3.] The regularity assumption $\partial_x u_0\in L^2$ in the uniqueness statement is not essential, it should be possible to have less regular  initial velocities,  as in \cite{MR4695791}. We do not pursue in this direction in the present paper.

     \item[4.] The boundary conditions in \eqref{equ-d} are used frequently in order to have a representation  of the pressure, this plays an important role in our analysis.    We leave the  domain cases $\mathbb{T}$ and $\mathbb{R}$ for future work.
 \end{itemize}
\end{Remark}

\bigbreak

 \subsection*{Notations} Throughout the paper,  we use  $\partial_x^{-1} f$ to denote the integral  $\int_0^x f( y) dy$ for $x\in(0,1).$
 We sometimes use $f'(s)$ to represent derivative of $f$ on $s.$   Note that $C(a, b, c, \cdots)$ denotes  a generic constant depending on $a, b, c, \cdots$ in the rest of this paper.   By $\|\cdot\|_{L^q}$, we mean $q-$power Lebesgue norms over $\Omega$.  We denote by $H^s$ the Sobolev space.


\section{A priori estimates}\label{S2}
 In this section,   we will establish some necessary a priori bounds for smooth solutions away from vacuum for IBVP \eqref{equ}.   A key point of  these estimates is that they are independent of the lower bound of density.

\subsection{Boundedness of density}
At first,   we provide a proof of the energy balance \eqref{balance}.
\begin{Lemma}\label{lemma1}
Consider $(\rho,  u)$ with $\rho>0$ a smooth solution  of  problem \eqref{equ}  defined on $[0, T]\times \Omega$ for some $T>0.$   Then for any $t\in [0, T],$ we have
\begin{align}
\int_\Omega [e(\rho)+\rho u^2 ](t, x)dx +\int_0^t \int_\Omega |\sqrt{\mu(\rho)}\pl_x u|^2(s, x) dx ds \leq C(\rho^*, E_0), \label{lem310}
\end{align}
where $C(\rho^*, E_0)$ is a positive constant depending on $ \rho^*$ and $E_0.$
\end{Lemma}
{\it Proof.} Multiply \eqref{equ-a} and \eqref{equ-b} by $e'(\rho)$ and $u$,respectively, and integrate with respect to space variable over $I$ to obtain
\begin{align}
\frac{1}{2}\frac{d}{dt}\int_\Omega [ 2 e(\rho)+\rho u^2] (t, x) dx +\int_\Omega  |\sqrt{\mu(\rho)}\pl_x u|^2(t, x)  dx=0,
\end{align}
where we used integration by parts and boundary conditions in \eqref{equ-d}.  Then we  integrate the above equality with respect to  time variable  to obtain \eqref{balance}.
 Because $e(\cdot)$ is continuous and $\rho_0$ is bounded,   we obtain \eqref{lem310} consequently.   \hfill  $\Box$

\begin{Lemma}\label{lemma2}
Assume that $(\rho, u)$  with $\rho>0$ is a smooth solution  of  problem \eqref{equ} defined on $[0, T]\times \Omega$ for some $T>0.$     We have the following upper bound for density:
\begin{align}\label{upperbdd}
\sup_{t\in [0, T]}\|\rho(t, \cdot)\|_{L^\infty}\leq C(E_0,   \rho^*, \mu_*,   \mu^*).
\end{align}
\end{Lemma}
{\it Proof.}
Given  that $\rho(t, x)$  is smooth and  positive,  we can rewrite \eqref{equ-a} as
\begin{align}
\pl_t \ln \rho +u\pl_x \ln \rho+\pl_x u=0. \label{lem1}
\end{align}
It follows from equation \eqref{equ-b} and boundary condition   \eqref{equ-d} that
\begin{align}
\mu(\rho) u_x-\rho u^2-p(\rho) =\pl_t(\pl_{x}^{-1} (\rho u)).\label{lem2}
\end{align}
Recall that  $\pl_{x}^{-1} (\rho u)$ denotes the integral  $\int_0^x (\rho u)(t, y) dy$ for $x\in(0,1).$
We then combine \eqref{lem1} and \eqref{lem2} to write
\begin{align}
&\pl_t\lt[ \mathcal{U}(\rho)  + \pl_{x}^{-1} (\rho u)\rt]+u\pl_x
\lt[ \mathcal{U}(\rho)  + \pl_{x}^{-1} (\rho u)\rt]=-p(\rho)
\end{align}
with \begin{equation}\label{def-mu}
\mathcal{U}(\rho)=\int_1^\rho \frac{\mu(s) }{s}ds.
\end{equation}
Next, we define a Lagrangian trajectory $X(t,  y)$ associated with  $u$  by
$$\frac{d}{dt}X(t,   y)= u(t, X(t, y)),\ \ X(t,  y)|_{t=0}=y.$$
Then we obtain from equation \eqref{lem2} and the  fact that the pressure is nonnegative  that
\begin{align*}
 \lt[ \mathcal{U}(\rho)  + \pl_{x}^{-1}(\rho u)\rt](t, X(t,  y))=&
\lt[ \mathcal{U}(\rho) + \pl_{x}^{-1}(\rho u)\rt](0, x)-\int_0^t p(\rho)(s, X(s,   y))ds\\
&\leq  \mathcal{U}(\rho_0) + \pl_{x}^{-1}(\rho_0 u_0).\end{align*}
Now,  if $\rho\leq 1$ for all $(t,  x)\in[0, T]\times \Omega,$ then we get \eqref{upperbdd}.   Otherwise,   we infer from   condition \eqref{mucon} that
\begin{align}\label{rho-rho0}
\mu_* \ln\rho(t, x)\leq  \mathcal{U}(\rho)  &\leq      \mathcal{U}(\rho_0) + \pl_{x}^{-1}(\rho_0 u_0)+\|\pl_{x}^{-1}(\rho u)\|_{L^\infty}\notag\\
 &\leq     \mu^* \ln{\rho^*}+ \|\sqrt{\rho_0}u_0\|_{L^2}\|\sqrt{\rho_0}\|_{L^2}+  \|\sqrt{\rho}u(t, \cdot)\|_{L^2}\|\sqrt{\rho(t, \cdot)}\|_{L^2}\notag\\
  &\leq    \mu^* \ln{\rho^*}+ 2E_0,
\end{align}
where  in the last inequality we used conservation of mass and energy balance \eqref{balance}.
Hence,   we obtained \eqref{upperbdd}.   The proof of this lemma is completed.
 \hfill $\Box$

\subsection{Lipschitz continuity of velocity for regular data} Here,   we focus on the first-order spatial derivative estimate for the velocity that   requires only  a boundedness of density.  To achieve it,  we need to introduce  the so-called effect viscous flux
$$G:=\mu(\rho)\partial_x u- p(\rho).$$

\begin{Lemma} \label{lemma3}
Consider  a smooth solution $(\rho, u)$ with $\rho>0$ defined on $[0, T]\times \Omega$   of  problem \eqref{equ}.
Then we have
\begin{multline}\label{lem311}
 \sup_{[0, T]}\| \pl_x u(t, \cdot)\|^2_{L^2} +\int_0^T \|\pl_x u(t, \cdot )\|^2_{L^\infty}\,dt +\int_0^T\lt( \|(\sqrt{\rho} \dot{u})(t, \cdot)\|^2_{L^2}+\|\partial_x G(t, \cdot)\|_{L^2}^2\rt)\, dt\\
 \leq  C(\|\pl_x u_0\|_{L^2}^2,    E_0,  \rho^*,  \mu_*^{-1}, \mu^*, T).
\end{multline}
\end{Lemma}
{\it Proof.}   It is clear that with the aid of \eqref{equ-a},    equation  \eqref{equ-b} can be rewritten  as
\begin{align}
\rho\dot{u}+\pl_x p(\rho)=\pl_x(\mu(\rho)\pl_x u).\label{lem310-1}
\end{align}
By taking $L^2$-inner product of \eqref{lem310-1} with $\dot{u},$ we see that
\begin{multline*}
 \frac{1}{2}\frac{d}{dt} \|\sqrt{\mu(\rho)}\pl_x u\|^2_{L^2}+\|\sqrt{\rho}\dot u\|^2_{L^2}-\int_I p\pl_x \dot{u} dx\\= \frac{1}{2}\int_\Omega \pl_t (\mu(\rho)) (\pl_x u)^2\,dx+\int_\Omega \pl_x(\mu(\rho)\pl_x u) (u \pl_x u) \,dx.
\end{multline*}
Notice that
\begin{align}
\pl_t (\mu(\rho)) +u\pl_x( \mu(\rho))+\mu'(\rho) \rho \pl_x u=0\label{murho-1}
\end{align}
and
\begin{align*}
\int_\Omega \pl_x(\mu(\rho)\pl_x u) (u \pl_x u) \,dx=\frac{1}{2}\int_\Omega \pl_x(\mu(\rho)) ( \pl_x u)^2 u \,dx-\frac{1}{2}\int_\Omega  \mu(\rho)(\pl_x u)^3 \,dx,
\end{align*}
we  find from above  that
\begin{multline}\label{lem311-2}
 \frac{1}{2}\frac{d}{dt} \|\sqrt{\mu(\rho)}\pl_x u\|^2_{L^2}+\|\sqrt{\rho}\dot u\|^2_{L^2}-\int_\Omega p\pl_x \dot{u} dx\\= -\frac{1}{2}\int_\Omega  \mu'(\rho)\rho  \pl_x u \,dx  -\frac{1}{2}\int_\Omega  \mu(\rho)(\pl_x u)^3 \,dx.
\end{multline}
Using  that
$$\int_\Omega p\pl_t \pl_x u dx=\frac{d}{dt}\int_\Omega p \pl_x u dx -\int_\Omega \pl_t p \pl_x u dx$$
and \eqref{equ-a},    we   rewrite \eqref{lem311-2} as
 \begin{multline}\label{lem311-3}
 \frac{1}{2}\frac{d}{dt} \|\sqrt{\mu(\rho)}\pl_x u\|^2_{L^2}+\|\sqrt{\rho}\dot u\|^2_{L^2}-\frac{d}{dt}\int_\Omega p \pl_x u dx\\
  =-\frac{1}{2}\int_\Omega \left( \mu'(\rho)\rho  \pl_x u  + \mu(\rho)(\pl_x u)^3 \right) \,dx+\int_\Omega p'(\rho) \rho(\pl_x u)^2 \, dx.
\end{multline}
Then,   with the aid of $p, \mu \in C^1(\mathbb{R}_+)$ and the boundedness of $\rho$ in \eqref{upperbdd},  we get from \eqref{lem311-2} that
\begin{align}
 \frac{1}{2}\frac{d}{dt} (\|\sqrt{\mu(\rho)}\pl_x u\|^2_{L^2}-2\int_\Omega p \pl_x u dx)+\|\sqrt{\rho}\dot u\|^2_{L^2}\leq&  C(\|\rho\|_{L^\infty})    (\|\pl_x u\|_{L^\infty}+1)\|\pl_x u\|^2_{L^2}.\label{lem311-5}
\end{align}

 At this stage,    it is necessary to estimate  pressure.   Thanks to \eqref{equ-a} and a basic energy estimate of $p(\rho),$  one  gets
\begin{align}
 \frac{d}{dt}\|p(\rho)\|_{L^2}^2 &\leq \|\pl_x u\|_{L^2}\|p^2(\rho)\|_{L^2} +\|p'(\rho)\rho p(\rho)\|_{L^2}\|\pl_x u\|_{L^2}\notag\\
&\leq C(\|\rho\|_{L^\infty}) \|\pl_x u\|_{L^2},\label{estp}
\end{align}
where   the boundedness of $\rho$ is used again.

Now we combine \eqref{lem311-5} and \eqref{estp} to  write
\begin{multline*}
 \frac{d}{dt}\lt(\|\sqrt{\mu(\rho)}\pl_x u\|^2_{L^2}+8{\mu_*}^{-1}\|p(\rho)\|^2_{L^2}-2 \int_\Omega  p \pl_x u dx\rt)+\|\sqrt{\rho} \dot u\|^2_{L^2}
 \\
  \leq C(\|\rho\|_{L^\infty}, \mu_*^{-1}) \lt((\|\pl_x u\|_{L^\infty}+1)\|\pl_x u\|^2_{L^2}+\|\pl_x u\|_{L^2}\rt)
\end{multline*}
Due to the fact that
\begin{align*}
&\|\sqrt{\mu(\rho)}\pl_x u\|^2_{L^2}+8\mu_*^{-1}\|p(\rho)\|^2_{L^2}-2 \int_\Omega p \pl_x u dx\\
&\geq \mu_*\|\pl_x u\|^2_{L^2}+8\mu_*^{-1}\|p(\rho)\|^2_{L^2} -\frac{1}{2}(4\mu_*^{-1} \|p\|^2_{L^2}+  \mu_*\|\pl_x u\|^2_{L^2})\geq \frac{1}{2}(\mu_*\|\pl_x u\|^2_{L^2}+\mu_*^{-1}\|p(\rho)\|_{L^2}^2),
 \end{align*}
it yields that
\begin{align}
&\mu_*\|\pl_x u\|^2_{L^2} +\int_0^t \|(\sqrt{\rho} \dot u)(s, x)\|^2_{L^2}ds\notag\\
&\leq \mu^*\|\pl_x u_0\|_{L^2}^2+     C(\|\rho\|_{L^\infty}, \mu_*^{-1})    \int_0^t  \lt((\|\pl_x u\|_{L^\infty}+1)\|\pl_x u\|^2_{L^2}+\|\pl_x u\|_{L^2}\rt)(s, x) ds.\label{Lem2.3000}
\end{align}

 Now,   we read \eqref{lem310-1} as
\begin{align*}
\pl_x(\mu(\rho)\pl_x u-p(\rho))=\rho \dot{u}
\end{align*}
and write that
\begin{align*}
\|\pl_x (\mu(\rho)\pl_x u- p)\|_{L^1} =\|\rho\dot{u}\|_{L^1}\leq&  \| {\rho} \|^{\frac{1}{2}}_{L^1}\|{\sqrt{\rho} \dot u}\|_{L^2}\\
\leq& \|{\sqrt{\rho} \dot u}\|_{L^2}
\end{align*}
This together with the fact that  $\dot W^{1, 1}(\Omega)\hookrightarrow L^\infty(\Omega)$ is valid for $G$,  the boundedness of $\rho$ and condition \eqref{mucon}  yield that
\begin{align}
\mu_*\| \pl_x u\|_{L^\infty}
\leq \|\mu(\rho) \pl_x u-p(\rho)\|_{L^\infty}  + \|p(\rho)\|_{L^\infty}\leq  \|\sqrt{\rho}\dot u\|_{L^2}+   C(\| \rho\|_{L^\infty}).\label{ellest}
\end{align}

Finally,  in virtue of \eqref{Lem2.3000},  \eqref{ellest} and \eqref{lem310} and  \eqref{upperbdd},  we conclude that
\begin{multline*}
  \sup_{[0, T]}\|\pl_x u(t, \cdot)\|^2_{L^2}+\int_0^T\|\pl_x u(t, \cdot)\|^2_{L^\infty}\,dt +\int_0^T \|(\sqrt{\rho} \dot u)(t, \cdot)\|^2_{L^2}ds +\int_0^T \|\partial_x G(t, \cdot)\|_{L^2}^2\,dx\\
 \leq C(\|\pl_x u_0\|_{L^2}^2,    E_0,  \rho^*,  \mu_*^{-1}, \mu^*, T).
\end{multline*}
This completes the proof of Lemma \ref{lemma3}.    \hfill $\Box$

\subsection{Time-weighted energy estimates}   It is clear that Lemma \ref{lemma3} requires a control of derivative of initial velocity,   and hence does not apply to merely finite-energy initial data.    In what follows,  we show that   it is possible to capture higher-order regularity estimates that  similar to  Lemma \ref{lemma3} by attaching additional time weights. This is due to the smoothing effect of  the momentum equation \eqref{equ-b} and is a key to the proof of  Theorem \ref{thm1}.

\begin{Lemma} \label{lemma4}
Consider  a smooth solution $(\rho, u)$ with $\rho>0$ defined on $[0, T]\times \Omega$   of  problem \eqref{equ}.
Then we have
\begin{multline}\label{lem411}
 \sup_{[0, T]} \| \sqrt{t}\pl_x u(t, \cdot)\|_{L^2}^2  +\int_0^T\lt( \|\sqrt{t}\pl_x u(t, \cdot )\|^2_{L^\infty} +\|(\sqrt{\rho\, t} \dot{u})(t, \cdot)\|^2_{L^2}+\|\sqrt{t}\partial_x G(t, \cdot)\|_{L^2}^2\rt)\, dt\\
 \leq  C( E_0,  \rho^*,  \mu_*, \mu^*, T).
\end{multline}
\end{Lemma}

{\it Proof.}   The proof of Lemma \ref{lemma4} is basically the same as for the Lemma \ref{lemma3},  just in the current case one needs to start estimation from the following time-weighted momentum equation:
\begin{align}
\sqrt{t}\rho\dot{u}+\pl_x (\sqrt{t}p(\rho))=\pl_x(\mu(\rho)\pl_x (\sqrt{t}u)).\label{lem410-1}
\end{align}
For reader's convenience,   we provide full details at here.    Precisely,   take $L^2$-inner product of \eqref{lem410-1} with $\sqrt{t}\dot{u},$ we obtain  that
\begin{align}
&\frac{1}{2}\frac{d}{dt} \int_{\Omega} \mu(\rho)|\pl_x(\sqrt{t} u)|^2 dx+\int_{\Omega} (\sqrt{t \rho} \dot u)^2 dx-\int_{\Omega} t p\pl_x \dot{u} dx\notag\\
=&\frac{1}{2}\int_\Omega \pl_t \mu(\rho)|\pl_x(\sqrt{t} u)|^2dx +\frac{1}{2}\int_\Omega  \mu(\rho)|\pl_xu|^2dx
+\int_\Omega \pl_x \mu(\rho)|\pl_x(\sqrt{t} u)|^2dx\notag\\
&-\frac{1}{2}\int_\Omega \mu(\rho) |\pl_x(\sqrt{t} u)|^2 \pl_x u dx.\notag
\end{align}
Using the fact \eqref{murho-1}, we rewrite the above  equality as follows
\begin{align}
&\frac{1}{2}\frac{d}{dt} \int_{\Omega} \mu(\rho)|\pl_x(\sqrt{t} u)|^2 dx+\int_{\Omega} (\sqrt{t \rho} \dot u)^2 dx-\int_{\Omega} t p\pl_x \dot{u} dx\notag\\
=& \frac{1}{2}\int_\Omega  \mu(\rho)|\pl_xu|^2dx
-\frac{1}{2}\int_\Omega (\mu(\rho)+\mu'(\rho)\rho) |\pl_x(\sqrt{t} u)|^2 \pl_x u dx.\label{lem0805}
\end{align}
Together with the fact  that
\begin{align*}
-\int_\Omega t p \pl_x \dot u dx=-\frac{d}{dt}\int_\Omega t p \pl_x u dx +\int_\Omega t \pl_tp \pl_x u dx +\int_\Omega p \pl_x udx
+\int_\Omega t \pl_x pu \pl_x udx
\end{align*}
and
\begin{align}
\pl_t p(\rho)+\pl_x( p(\rho)) u+p'(\rho) \rho \pl_x u=0,\label{preequ}
\end{align}
we deduce from \eqref{lem0805} that
\begin{align}
&\frac{1}{2}\frac{d}{dt} \int_{\Omega}\lt( \mu(\rho)|\pl_x(\sqrt{t} u)|^2  -2 t p \pl_x u\rt)dx+\int_{\Omega} (\sqrt{t \rho} \dot u)^2dx\notag\\
=& \frac{1}{2}\int_\Omega  \mu(\rho)|\pl_xu|^2dx
-\frac{1}{2}\int_\Omega (\mu(\rho)+\mu'(\rho)\rho) |\pl_x(\sqrt{t} u)|^2 \pl_x u dx\notag\\
&-\int_\Omega p \pl_x u dx +\int_\Omega p'(\rho)\rho |\pl_x(\sqrt{t} u)|^2 dx.\label{lem0805-1}
\end{align}
The basic energy estimates for the pressure equation $\sqrt{t}\times\eqref{preequ}$ gives
\begin{align}
\frac{d}{dt}\|\sqrt{t} p(\rho)\|_{L^2}^2\leq &\|p(\rho)\|_{L^2}^2+\|p^2(\rho)\|_{L^2}\|\sqrt{t}\pl_x(\sqrt{t} u)\|_{L^2}\notag\\
&+\|p(\rho)p'(\rho)\rho\|_{L^2}\|\sqrt{t}\pl_x(\sqrt{t}u)\|_{L^2}. \label{tpenergy}
\end{align}
With the aid of \eqref{lem310} and \eqref{upperbdd}, we combine \eqref{lem0805-1} and \eqref{tpenergy} to get
\begin{align}
&\|\pl_x(\sqrt{t}u)\|^2_{L^2} +\int_0^t \|(\sqrt{s \rho} \dot u)(s,x)\|^2_{L^2}ds \notag\\
\leq & C(\mu_*^{-1})E(0)+ C(\mu_*^{-1}, \mu^*) \int_0^t \lt( \|p(\rho)\|_{L^2}^2 +|\pl_x(\sqrt{s}u )|_{L^\infty}s^{-1/2}\|\pl_x(\sqrt{s}u)\|_{L^2}^2 \rt) ds \notag\\
 &+C(\mu_*^{-1}, \mu^*) \int_0^t\lt((1+s) \|\pl_x(\sqrt{s}u)\|_{L^2}^2+ \|\pl_x(\sqrt{s}u)\|_{L^2}\rt) ds.\label{plxtul2}
\end{align}
Moreover, similar to derive \eqref{ellest}, we have
\begin{align}
\mu_*\|\pl_x(\sqrt{t}u)\|_{L^\infty}&\leq \|\mu(\rho)\pl_x(\sqrt{t} u)-\sqrt{t}p(\rho)\|_{L^\infty}+\|\sqrt{t}p(\rho)\|_{L^\infty}\notag\\
&\leq \|\sqrt{t \rho} \dot u\|_{L^2}+C(\mu^*)\sqrt{t}.\label{plxtulinfty}
\end{align}

Now, with the help of \eqref{plxtul2}, \eqref{plxtulinfty} and Gronwall's inequality, we prove the inequality \eqref{lem411}.
The proof of Lemma \ref{lemma4} is completed.   \hfill $\Box$

 Roughly speaking,   Lemma \ref{lemma4} and Lemma \ref{lemma3}   indicate that weak solutions will become regular instanously.   As such,   one can try to  find higher-order estimates  under a better control.    The following lemma plays an important role on the compactness result of   density.

\begin{Lemma}\label{wlem}
Consider  a smooth solution $(\rho, u)$  with $\rho>0$ defined on $[0, T]\times \Omega$   of  problem \eqref{equ}. Then,  for all $t_0\in[0,  T]$ we have
\begin{align}
\sup_{[t_0, T]} \|\sqrt{ t\rho}\dot u(t, \cdot)\|_{L^2}^2  + \int_{t_0}^T  \|\sqrt{t}\partial_x \dot{u}(t, \cdot)\|_{L^2}^2\,dt \leq   C(\|\partial_x u(t_0)\|_, E_0,  \rho^*,  \mu_*, \mu^*, T).\label{lem-w}
\end{align}
\end{Lemma}
{\it Proof.}  At first,  multiplying  the momentum equation \eqref{equ-b} by $\sqrt{t}$ and taking material derivative on the resulting equation yield that
\begin{multline}\label{mdu}
\rho\frac{D}{Dt}\lt( \sqrt{t}\dot{u}\rt)   +\pl_x\lt(\sqrt{t}\dot{p}\rt)-\pl_x \lt( \mu(\rho)\sqrt{t}\pl_x\dot{u}\rt)-\frac{1}{2\sqrt{t}}\rho\dot{u}+\lt(\sqrt{t}\pl_x u\rt)\rho\dot{u}\\
=-\dot{\rho} \lt( \sqrt{t}\dot{u}\rt)+\pl_x \lt(\lt(\sqrt{t}\pl_x u\rt)  \frac{D}{Dt}\mu(\rho)  -\mu(\rho)\sqrt{t}(\pl_x u)^2 \rt).
\end{multline}
We integrate the product of \eqref{mdu} and $\sqrt{t} \dot u$ over $\Omega$ and use  boundary condition \eqref{equ-d}      to  write
\begin{align}
 \frac{1}{2}\frac{d}{dt}\|\sqrt{\rho\,t}\dot{u}\|_{L^2}^2+\|\sqrt{t\,\mu(\rho)}\pl_x \dot{u}\|_{L^2}^2=R_1+\cdots R_5,\label{mdu1}
\end{align}
where
\begin{align*}
&R_1=  \int_\Omega      \pl_x\lt(  \sqrt{t}\dot{u}\rt)\,\lt(\sqrt{t}\dot{p} \rt)\,dx,\\
&R_2= \frac{1}{2 }\int_\Omega  \rho(\dot{u} )^2   \,dx,\\
&R_3=-\int_\Omega   \lt(t\pl_x u\rt)\rho(\dot{u} )^2  \,dx,\\
&R_4=-\int_\Omega  \dot{\rho} \lt( \sqrt{t}\dot{u}\rt)^2    \,dx,\\
&R_5=-\int_\Omega  \pl_x\lt(     \sqrt{t}\dot{u}   \rt) \lt(\sqrt{t}\pl_x u\rt)  \frac{D}{Dt}\mu(\rho)   \,dx,\\
&R_6=\int_\Omega  \pl_x\lt(     \sqrt{t}\dot{u}   \rt)    \mu(\rho)\sqrt{t}(\pl_x u)^2    \,dx.
\end{align*}
Next, we will deal with the remainder  terms.    Firstly,    notice that
\begin{align*}
R_2=\frac{1}{2}\|\sqrt{\rho} \dot{u}\|_{L^2}^2
\end{align*}
and
\begin{align*}
|R_3|\leq      \|\partial_x u\|_{L^\infty}\,\|\sqrt{t \rho} \dot{u}\|_{L^2}^2
\end{align*}
and
\begin{align*}
|R_6| \leq  \|\sqrt{\mu(\rho)}\|_{L^\infty}  \, \|\partial_x u\|_{L^\infty}\,         \|\sqrt{t}\partial_x u\|_{L^2}          \|\sqrt{t \,\mu(\rho)}\,\pl_x \dot{u}\|_{L^2}.
\end{align*}
To handle $R_1, ~R_4,~R_5$,    we take use of  equation \eqref{equ-a}  frequently and write that
\begin{align*}
|R_1|=&     \int_\Omega    \lt|  \pl_x\lt(  \sqrt{t}\dot{u}\rt)\,\sqrt{t}\lt(   p'(\rho)\,\rho\,\partial_x u       \rt) \rt|\,dx\\
\leq    &   \|p'(\rho)\,\rho\|_{L^\infty}    \, \|\sqrt{t  }\,\pl_x \dot{u}\|_{L^2}\,   \|\sqrt{t}\partial_x u\|_{L^2},
\end{align*}
\begin{align*}
|R_4|=      \int_\Omega    \lt|  \rho\,\pl_x u \rt| \lt( \sqrt{t}\dot{u}\rt)^2   \,dx
\leq      \|\partial_x u\|_{L^\infty}\,    \|\sqrt{t\,\rho} \dot{ u}\|_{L^2}^2
\end{align*}
and
\begin{align*}
|R_5|=&      \int_\Omega \lt|   \pl_x\lt(     \sqrt{t}\dot{u}   \rt)  \lt(\sqrt{t}\pl_x u\rt)  (\mu'(\rho)\,\rho\,\pl_x u)\rt|    \,dx \\
\leq&     \|\mu'(\rho)\,\rho\|_{L^\infty}\,  \|\partial_x u\|_{L^\infty}\,        \|\sqrt{t }\pl_x \dot{ u}\|_{L^2} \, \|\sqrt{t}\partial_x u\|_{L^2}.
\end{align*}

Taking into account of \eqref{mdu1} and  above estimates of $R_1, \cdots, R_6,$ it follows that
\begin{multline}\label{es-lem5}
 \frac{d}{dt}\|\sqrt{\rho\,t}\dot{u}\|_{L^2}^2+ \|\sqrt{t\,\mu(\rho)}\pl_x \dot{u}\|_{L^2}^2\\
 \leq   \|\sqrt{\rho} \dot{u}\|_{L^2}^2+   \|\partial_x u\|_{L^\infty}\,\|\sqrt{t \rho} \dot{u}\|_{L^2}^2  + C(\|\rho\|_{L^\infty}, \mu_*) \lt(   \|\partial_x u\|_{L^\infty}^2+1 \rt) \|\sqrt{t}\partial_x u\|_{L^2}^2,
\end{multline}
where we used frequently the Young inequality and conditions \eqref{pcon}, \eqref{mucon}.

Finally,   we apply Lemma \ref{lemma2}-\ref{lemma4} to \eqref{es-lem5} and use  Gr\"onwall's  inequality to conclude  \eqref{lem-w}.    The proof of this lemma is completed. \hfill $\Box$

\section{Proof of existence part}
This section is devoted to the proof of existence parts in Theorem \ref{thm1}.    
In a first step,    we need to construct a sequence of approximate solutions by using some ideas developed by Jiang-Xin-Zhang in \cite{MR2254008}
for one-dimensional compressible Navier-Stokes equations with free boundary.   Note that the boundary condition in \cite{MR2254008} is of the same type as \eqref{equ-d}.   Then,   we will focus on establishing the compactness property for   the  global approximate solutions,   in particular the one for the density that stated in Proposition  \ref{Lem-compa1}.
To begin with,     we have to modify the initial data to fit  the requirement on the regularity and make the density strictly positive.
\subsection*{ Step 1.   Modified initial data}
 We denote by $j_\ep(x)$ the Friedrichs mollifier.  Let $\psi(x)\in\mathcal{C}^\infty_0(\R)$ satisfy $\psi(x)\equiv 1$ when $|x|\leq \frac{1}{2}$ and $\psi(x)\equiv 0$ when $|x|\geq 1,$ and define $\psi_\ep(x):=\psi(\frac{x}{\ep}).$ For simplicity we still denote by $(\rho_0, u_0)$ the extension of  $(\rho_0, u_0)$ in $\R$,  i.e.
\begin{equation}
\rho_0(x)=\left\{\begin{aligned}
&\rho_0(1),\quad &&\ \  x\in(1, \infty),\\
&\rho_0(x), \quad &&\ \  x\in[0, 1],\\
&\rho_0(0),\quad &&\ \ x\in(-\infty, 0)
\end{aligned}\right. \andf
u_0(x)=\left\{\begin{aligned}
& u_0(x), \quad &&\ \ x\in[0, 1],\\
& 0,\quad &&\ \ {\rm otherwise}.
\end{aligned}\right.
\end{equation}
   We define the approximate initial data to $(\rho_0,  u_0):$
   \begin{align*}
   \rho_0^\ep(x):=& (\rho_0\star j_\ep)(x)+\ep,\\
   u_0^\ep(x):=&   (u_0\star j_\ep)(x)[1-\psi_{\ep}(x)-\psi_{\ep}(1-x)] +(u_0\star j_\ep)(0)\,\psi_{\ep}(x)+  (u_0\star j_\ep)(1)\,\psi_{\ep}(1-x)\\
   &\quad +\frac{p(\rho_0^\ep(x))}{\mu(\rho_0^\ep(x))}\int_0^x \psi_{\ep}(y)\,dy+  \frac{p(\rho_0^\ep(x))}{\mu(\rho_0^\ep(x))}\int_{x}^1 \psi_{\ep}(1-y)\,dy.
   \end{align*}
  Then $\rho_0^\ep\in\mathcal{C}^{1+s}([0, 1]),$  $u_0^\ep\in\mathcal{C}^{2+s}([0, 1])$   for any $s\in(0, 1),$  and $ \rho^\ep_0$ and $u_0^\ep$ are compatible with the boundary conditions \eqref{equ-d},  such that as $\ep\to 0$ we have
  \begin{equation}
  \rho_0^\ep\to \rho_0 \quad  {\rm in}~L^2(\Omega) \andf u^\ep_0\to u_0 \quad  {\rm in}~L^2(\Omega).
\end{equation}

\subsection*{Step 2.  Global existence of approximate   smooth solutions}
We are going to solve \eqref{equ}  with data $(\rho_0^\epsilon, u_0^\epsilon).$  Since  the viscosity $\mu(s)\geq \mu_*>0$ for every $s\geq 0,$  the following global in time existence result is in the same way as that of \cite[Page 246]{MR2254008}.

\begin{Proposition}\label{prolocal}
Consider the {\rm IBVP} \eqref{equ} with initial data $(\rho_0, u_0)$ replaced by $(\rho_0^\epsilon, u_0^\epsilon).$
There   exists a unique global in time  regular solution $(\rho^\epsilon, u^\epsilon)$  on  $[0, T]$ for any $T>0,$ such that $\rho^\epsilon(t, x)>0$ and
\begin{align}
u^\epsilon\in L^\infty(0, T; H^1(\Omega))\cap  L^2(0, T; H^2(\Omega)) ; \ \ \rho^\epsilon \in L^\infty(0, T; H^1(\Omega)).\label{local}
\end{align}
\end{Proposition}

 Since $(\rho^\epsilon, u^\epsilon)$ is smooth enough and away from  vacuum,    it satisfies all the estimates given in Lemmas \ref{lemma1}--\ref{wlem}  with all the same controls independent of $\epsilon.$   In particular,
  \begin{multline}\label{es-main1}
\sup_{t\in[0,  T]}\|(\sqrt{\rho}u  ,\sqrt{t} \pl_x u^\epsilon,)(t, \cdot)  \|_{L^2}^2 +\sup_{t\in[0,  T]}\|\rho(t, \cdot)\|_{L^\infty}+ \int_{0}^{T} \| \partial_x u^\epsilon\|_{L^2}^2 \,dt\\
\leq  C(E_0,   \rho^*,  \mu_*,   \mu^*, T, )
\end{multline}
 and for any $t_0>0$
 \begin{multline}\label{es-main2}
\sup_{t\in[t_0,  T]}\|(   \pl_x u^\epsilon,  \sqrt{ \rho^\epsilon}\dot{u}^\epsilon)(t, \cdot)  \|_{L^2}^2+ \int_{t_0}^{T} \lt(\| \partial_x u^\epsilon\|_{L^\infty}^2+ \|( \pl_x G^\epsilon,   \pl_x\dot{u}^\epsilon)(t, \cdot)\|_{L^2}^2\rt)\,dt\\
\leq  C(E_0,   \rho^*,  \mu_*,   \mu^*, T, t_0),
\end{multline}
   where $G^\ep=\mu(\rho^\ep)\pl_x u^\ep- p(\rho^\ep).$

\subsection*{Step 3. Strong convergence}    At this stage,   we are ready to establish robust compactness property for the sequence of approximate solutions.           We first address on the velocity component.

Owing to the time-weighted  estimates in\eqref{es-main2} and $u|_{\pl\Omega}=\dot u^\ep|_{\pl \Omega}=0$, it follows that for any $t_0,\,T>0,$
$$ \dot u^\ep\in L^2(t_0, T; H^1(\Omega))\andf u^\ep\in L^{\infty}(t_0, T; H^1(\Omega)).$$
Then, with the aid of \eqref{es-main2} again,  we have $$ \pl_t u^\ep \in L^2(t_0, T; L^2(\Omega)),$$   which together with the Aubin-Lions lemma gives that,  up to a subsequence
\begin{equation}\label{uconinl2}
\left\{\begin{aligned}
& u^\ep\rightarrow u  \ \ {\rm weakly~ in}\ \ L^2(0, T; \Omega),\\
& u^\ep\rightarrow u \ \ {\rm strongly~ in}\ \ L^2(t_0, T; \Omega),
\end{aligned}\right.
\end{equation}
for  some $u$ satisfies \eqref{es-main1} and \eqref{es-main2}.

We are now in a good position to prove the strong convergence of the density.  We deduce from \eqref{es-main1} that, up to a subsequence,
$$\rho^\ep\to \rho\quad{\rm weakly^*~ in}\ \ L^\infty(0, T; \Omega).$$
But this is not enough to pass to the limit in the pressure term of the momentum equation. To achieve that purpose,  one can prove  some strong convergence property of the effective viscous flux $G^\ep:=\mu(\rho^\ep)\pl_x u^\ep-p^\ep.$  That is,
\begin{align}
G^\ep\rightarrow G\ \ {\rm strongly~ in}\ \ L^2(t_0, T; \Omega). \label{Gcon}
\end{align}
Indeed, from \eqref{es-main1} and \eqref{es-main2}, one has
\begin{align}G^\ep \in L^\infty(t_0, T; L^2(\Omega))\cap L^2(t_0, T; H^1(\Omega)).\label{gbasic}\end{align}
 Thanks to the  indentity
\begin{align}
G^\ep_t=&-\mu'(\rho^\ep)(\pl_x(\rho^\ep u^\ep))\pl_x u^\ep+\mu(\rho^\ep)(\pl_x \dot u^\ep-\pl_x(u^\ep \pl_x u^\ep))-p^\ep_t\notag\\
=&(-\pl_x(\mu(\rho^\ep) u^\ep)+(\mu(\rho^\ep)-\mu'(\rho^\ep)\rho^\ep))\pl_x u^\ep\notag\\
&+\mu(\rho^\ep)(\pl_x \dot u^\ep-\pl_x(u^\ep \pl_x u^\ep))+\pl_x(p(\rho^\ep) u^\ep)-(p(\rho^\ep)-p'(\rho^\ep)\rho^\ep)\pl_x u^\ep,\label{gtequ}
\end{align}
one can further deduce  that
\begin{align}
 G^\ep_t\in L^2(t_0, T; H^{-1}(\Omega)),
\end{align}
which together with \eqref{gbasic} and  the Aubin-Lions lemma gives \eqref{Gcon}.
 Now,  we have
 \begin{Proposition}\label{Lem-compa1}
Up to a subsequence,  $\rho^\ep\to \rho$ strongly in $L^1(t_0, T;   L^1(\Omega))$.
\end{Proposition}
\begin{proof}   The previous consideration of the approximate solution sequence  ensures that there exists a subsequence $(\rho^\epn, u^\epn)_{n\in \mathbb{N}}$ of $(\rho^\ep, u^\ep)$ such that
\begin{align*}
\rho^\epn\rightharpoonup^*\rho\ \ {\rm in}\ \ L^\infty(0, T; \Omega)\ \ {\rm and}\ \ u^\epn \rightarrow u \ \ {\rm in}\ \ L^2(t_0, T; \Omega).
\end{align*}
Of course, for all $n \in \mathbb{N},$ it satisfies
\begin{align}
\rho_t^\epn+\pl_x(\rho^\epn u^\epn)=0,\label{ruepn}
\end{align}
and the limit $(\rho, u)$ satisfies
\begin{align}
\rho_t+\pl_x(\rho u)=0.\label{rulimit}
\end{align}
Applying Theorem II.3 of \cite{MR1022305}, it makes sure that $\rho$ is actually a renormalized solution of \eqref{rulimit} and satisfies  (recall from \eqref{def-mu} that $\mathcal{U}(\rho)=\int_1^\rho \mu(s)\,ds$)
\begin{align}
\pl_t(\rho\, \mathcal{U}(\rho))+\pl_x(u\, \mathcal{U}(\rho) )+\rho\,\mu(\rho)\pl_x u=0.\label{rensolu}
\end{align}
So the $(\rho^\epn, u^\epn)$ chosen in \eqref{ruepn} also implies
\begin{align}
\pl_t(\rho^\epn \mathcal{U}(\rho^\epn))+\pl_x(  u^\epn  \mathcal{U}(\rho^\epn))+\rho^\epn \mu(\rho^\epn)\pl_x u^\epn=0.
\end{align}
Next,  we replace $\mu(\rho^\epn)\pl_x u^\epn$, according to the definition of $G^\epn,$ to have
\begin{align}
\pl_t(\rho^\epn \mathcal{U}(\rho^\epn))+\pl_x(  u^\epn  \mathcal{U}(\rho^\epn))+ \rho^\epn\,p(\rho^\epn)+ \rho^\epn G^\epn=0,
\end{align}
whose  limit version is
\begin{align}
\pl_t(\overline{\rho  \,\mathcal{U}(\rho )})+\pl_x(  u \,  \mathcal{U}(\rho ))+ \overline{\rho \,p(\rho )}+ \rho\,  G =0,
\end{align}
due to  the convergence property \eqref{Gcon}.  Finally,   one can make use of the convex property of functions $z\rightarrow s\, \mathcal{U}(s)$  and  $s \rightarrow s\,p(s)$ by conditions \eqref{pcon} and \eqref{mucon},  and process a standard compactness argument as in \cite{MR4642822} to prove that
$\rho^\epn \rightarrow \rho  \ \ {\rm pointwise,}$
and thus the resut of the Proposition \ref{Lem-compa1}.
\end{proof}

\subsection*{Step 4.  Passing to the limit}
One can  pass to the limit in all the nonlinear terms of the momentum equation to conclude that $(\rho, u)$ is a solution to \eqref{equ}.    Particularly,     the following equation  is satisfied in the sense of distributions
\begin{align}
  \pl_t(\pl_{x}^{-1} (\rho u))+\rho u^2 =\mu(\rho) u_x-p(\rho) .
\end{align}
Moreover,    the conservation of mass and boundedness of density  ensure that $\rho\in C(R_+; L^p)$ for all $p<\infty.$            The solution  $(\rho, u)$  satisfy   the regularity properties stated  in Theorem \ref{thm1} thanks to \eqref{es-main1},  \eqref{es-main2},   and also  the  regularity  statement for the case that    $\partial_x u_0\in L^2$   additionally.
We complete the proof of existence part of Theorem \ref{thm1}.
\hfill $\Box$

\section{Proof of uniqueness statement}\label{S4}
In this section, we focus on the uniqueness of of solutions to problem \eqref{equ} given the initial data $0\leq \rho_0(x)\leq \rho^*, \sqrt{\rho_0} u_0\in L^2(\Omega)$ and  $\pl_x u_0 \in L^2(\Omega)$ additionally.  To avoid the loss of one derivative in
the stability analysis,  it is convenient to  perform a Lagrangian change of coordinates for the system \eqref{equ}. First of all, we introduce the Lagrangian flow map $\xi(y,t)$ as follows:
\begin{equation*}
\begin{cases}
&\pl_t\xi(y,t)=u(\xi(y,t),t),\\
&\xi(y,0)=y.
\end{cases}
\end{equation*}
 Then we can define the following terms
\begin{equation*}
\begin{aligned}
  &\tilde{\rho}(y,t):=\rho(\xi(y,t),t),\quad \tilde{u}(y,t):=u(\xi(y,t),t),\\
   &\tilde{p}(y,t):=p(\xi(y,t),t),\quad \tilde{\mu}(y,t):=\mu(\xi(y,t),t).
   \end{aligned}
\end{equation*}
Furthermore,  we denote a function $\eta=\eta(y,t)$ as
\begin{align*}
\eta(y,t)=\xi_y(y,t)=1+\int_0^t\pl_y\tilde{u}(y,\tau)d\tau,
\end{align*}
and rewrite the system \eqref{equ} in the Lagrangian coordinates as follows
\begin{equation}\label{refomulate-0}
\left\{
\begin{array}{rl}
\pl_t\tilde{\rho}+\frac{\pl_y\tilde{u}}{\eta}\tilde{\rho}&=0,\\[0.5mm]
\pl_t\tilde{\rho}\tilde{u}+\frac{\pl_y\tilde{p}}{\eta}&=\frac{1}{\eta}\pl_y\left(\frac{\tilde{\mu}\pl_y\tilde{u}}{\eta}\right).
\end{array}
\right.
\end{equation}
Since $\pl_t\eta=\pl_y\tilde{u}$, it holds from $\eqref{refomulate-0}_1$ that
$$\pl_t(\eta\tilde{\rho})=\pl_y\tilde{u}\tilde{\rho}-\eta\frac{\pl_y\tilde{u}}{\eta}\tilde{\rho}=0.$$
Setting $\tilde{\rho}(y,0)=\rho_0$ and combining with the fact that $\eta(y,0)=1$, we obtain
$$\eta\tilde{\rho}=\rho_0.$$

Next, the equations \eqref{refomulate-0} can be expressed in terms of $(\eta,\tilde{u})$ (denote still by $(\eta,u)$ without confusing) in variables $(y,t)$:
\begin{equation}\label{refomulate-1}
\left\{
\begin{array}{rl}
\pl_t\eta&=\pl_yu,\\[0.5mm]
\rho_0\pl_tu+\pl_yp\left(\frac{\rho_0}{\eta}\right)&=\pl_y\left(\mu\left(\frac{\rho_0}{\eta}\right)\frac{\pl_yu}{\eta}\right).
\end{array}
\right.
\end{equation}
with initial boundary condition
\begin{align}
&(\eta, u)(x, 0)=(1, u_0)(x), &&\ \ {\rm in}\ \ \Omega,\label{lag-ini}\\
& u=0,\ \ \ p\left(\frac{\rho_0}{\eta}\right)-\mu\left(\frac{\rho_0}{\eta}\right)\frac{\pl_y u}{\eta}=0, &&\ \ {\rm on}\ \ \pl \Omega.\label{lag-bd}
\end{align}
We point out that the system \eqref{refomulate-1} is equivalent to \eqref{equ} whenever, say,
\begin{align}
  \int_0^T\|\pl_y \tilde{u}\|_{L^{\infty}}dt<k_0<1
\end{align}
for sufficiently small $T$. Actually, we have
\begin{equation*}
\pl_y \tilde{u}(y,t)=\pl_x u(x,t)\left(1+\int_0^t\pl_y\tilde{u}(y,\tau)d\tau\right),
\end{equation*}
which yields
\begin{align}
\int_0^T\|\pl_y \tilde{u}\|_{L^{\infty}}dt\leq \frac{\int_0^T\|\pl_x u\|_{L^{\infty}}dt}{1-\int_0^T\|\pl_x u\|_{L^{\infty}}dt}<k_0<1
\end{align}
for sufficiently small $T$ by using Lemma \ref{lemma3}.

With the above preparation, now we tackle the proof of uniqueness. Let $(\eta_1,u_1)$ and $(\eta_2,u_2)$ be two solutions of system \eqref{refomulate-1}, subject to \eqref{lag-ini}-\eqref{lag-bd}, satisfying the regularities stated in
Theorem \ref{thm1}, on $[0,T_0]$, with the same initial data. Let $\delta u=u_1-u_2$ and $\delta\eta=\eta_1-\eta_2$. According to \eqref{refomulate-1}, $\delta u$ satisfies the
following equation:
\begin{align}\label{delta-u}
\rho_0\pl_t\delta u+\pl_y\left(p_1-p_2\right)=\pl_y\left(\mu_1\frac{\pl_yu_1}{\eta_1}-\mu_2\frac{\pl_yu_2}{\eta_2}\right),
\end{align}
where $p_i=p\big(\frac{\rho_0}{\eta_i}\big)$ and $\mu_i=\mu\big(\frac{\rho_0}{\eta_i}\big)$ for $i=1,2$.
Multiplying \eqref{delta-u} by $\delta u$, integrating the resultant over $\Omega$, and integrating
by parts, we have
\begin{align*}
&\frac12\frac{d}{dt}\int_{\Omega}\rho_0(\delta u)^2 dx+\int_{\Omega}\mu_2(\pl_y\delta u)^2dx+\int_{\Omega}\mu_2\frac{(\pl_y\delta u)^2(1-\eta_1)}{\eta_1}dx\\
=&-\int_{\Omega}(\mu_1-\mu_2)\frac{\pl_y u_1\pl_y\delta u}{\eta_1}dx-\int_{\Omega}\mu_2\frac{\eta_2-\eta_1}{\eta_1\eta_2}\pl_y u_2\pl_y\delta udx+\int_{\Omega}(p_1-p_2)\pl_y\delta udx.
\end{align*}
Notice that
\begin{align*}
  \|\eta_2-\eta_1\|_{L^{2}}\leq \int_0^t||\pl_y\delta u\|_{L^{2}}d\tau,\quad \forall t\in[0,T_0],
\end{align*}
which yields
\begin{align}
 \label{mu12} \|\mu_1-\mu_2\|_{L^{2}}\leq C\|\rho_0\|_{L^{\infty}}||\eta_2-\eta_1\|_{L^{2}}\leq C\int_0^t||\pl_y\delta u\|_{L^{2}}d\tau,\quad \forall t\in[0,T_0],\\
  \label{p12}\|p_1-p_2\|_{L^{2}}\leq C\|\rho_0\|_{L^{\infty}}||\eta_2-\eta_1\|_{L^{2}}\leq C\int_0^t||\pl_y\delta u\|_{L^{2}}d\tau,\quad \forall t\in[0,T_0].
\end{align}
It follows from the Young and Sobolev embedding inequalities, \eqref{mu12}-\eqref{p12}, that
\begin{align*}
&\frac{d}{dt}\int_{\Omega}\rho_0(\delta u)^2 dx+\int_{\Omega}(\pl_y\delta u)^2dx\\
\leq& C\int_0^t\|\pl_y\delta u\|_{L^{2}}d\tau(\|\pl_y u_1\|_{L^{\infty}}+\|\pl_y u_2\|_{L^{\infty}}+1)\left\|\pl_y\delta u\right\|_{L^{2}}\notag\\
&+C\left\|\pl_y\delta u\right\|_{L^{2}}^2\int_0^t||\pl_y u_1\|_{L^\infty}d\tau.
\end{align*}
Finally, for sufficiently small $T_0$ such that
$$C\int_0^t \|\pl_y u_1\|_{L^\infty} d\tau\leq \frac{1}{4},$$
we use Cauchy inequality to obtain
\begin{align}
&\frac{d}{dt}\int_{\Omega}\rho_0(\delta u)^2 dx+\int_{\Omega}(\pl_y\delta u)^2dx\notag\\
\leq& Ct\int_0^t\|\pl_y\delta u\|_{L^2}^2d\tau\left(1+\|\pl_y u_1\|_{L^{\infty}}^2+\|\pl_y u_2\|_{L^{\infty}}^2\right),
\end{align}
for some constant $C$ that depends on $k_0, M_0, E_0, \mu^* $ and $\mu_*^{-1}$.

In order to prove the uniqueness on $[0,T_0]$ by applying Gronwall's Lemma, it is crucial to have
\begin{align}
\int_0^{T_0}t\|\pl_y u_i\|_{L^{\infty}}^2dt<+\infty,\quad i=1,2,
\end{align}
which is equivalent to
\begin{align}
\int_0^{T_0}t\|\pl_x u_i\|_{L^{\infty}}^2dt<+\infty,\quad i=1,2.
\end{align}
Note that this is guaranteed by Lemma \ref{lemma4}, so the uniqueness of solutions is proved.

\section{Long-time behavior of rough solutions}
In this section,   we  first  derive the decay estimates of $\|e\|_{L^1}$ and $\|\sqrt{\rho}u\|_{L^2}$ in the following proposition.
\begin{Proposition}\label{decay1}
Consider weak solutions to system \eqref{equ} subject to initial conditions \eqref{initialdata-Cond1} and \eqref{initialdata-Cond2}. The following estimate holds:
\begin{equation}\label{decay1-est}
  \|e(t, \cdot)\|_{L^1}+\| u(t, \cdot)\|_{L^2}\leq Ce^{-\alpha t},\quad \forall t\geq t_0>0,
\end{equation}
for some constants $C,\alpha$ depending only on $E_0, \mu_*, \mu^*, \rho^*$.
\end{Proposition}
\begin{proof}
Recall from Lemma \ref{lemma2} that $0\leq \rho\leq \overline{\rho}:=\|\rho\|_{L^\infty(\R_+\times \Omega)}$ with $\overline{\rho}$ depending only on $E_0, \rho^*, \mu_*, \mu^*,$ and from Lemma \ref{lemma1} that
\begin{align}\label{recal-1}
\frac{1}{2}\frac{d}{dt}\int_\Omega [ 2 e(\rho)+\rho u^2] dx +\int_\Omega\mu(\rho) |\pl_x u|^2 dx=0.
\end{align}
Define $a:=\rho-1.$ Since $p(\rho)$ is nonnegative,  nondecreasing and convex, then there exists a universal positive constant $C$ such that    $a \,p\geq C( a+a^2),$ and thus  $\|\rho(t,\cdot)\|_{L^1(\Omega)}=1$ implies that
\begin{align*}
 C \int_{\Omega}a^2dx\leq \int_{\Omega}a\,p\,dx.
\end{align*}
Since the equation $p=-\partial_x^{-1}(\rho u)_t-\rho u^2+\mu(\rho) u_x$ is satisfied, we obtain
\begin{align}\label{es-500}
  \int_{\Omega}apdx&=-\frac{d}{dt}\int_{\Omega}\partial_x^{-1}(\rho u)a\,dx-\int_{\Omega}\partial_x^{-1}(\rho u)\partial_x(\rho u)\,dx-\int_{\Omega}a\rho u^2\,dx+\int_{\Omega}\mu(\rho) a\partial_xu\,dx\notag\\
  &\leq -\frac{d}{dt}\int_{\Omega}\partial_x^{-1}(\rho u)a\,dx+(1+2\overline{\rho})\|\sqrt{\rho} u\|_{L^2}^2+\frac{C}{2}\|a\|_{L^2}^2+\frac{\mu^*}{2C}\left\|\sqrt{\mu(\rho)}\partial_x u\right\|_{L^2}^2.
\end{align}
In order to control the second term on the right-hand side of the above inequality, we notice the following identity for the mean of velocity $\bar u(t):=\int_\Omega u(t, x)\,dx:$
\begin{equation*}
    \bar u(t)= \int_\Omega \rho(t, x) \bar u(t)\,dx -\int_\Omega (\rho u)(t, x)\,dx,
\end{equation*}
which  together with Poincar\'{e}'s inequality implies that
\begin{equation}\label{decay000}
    |\bar u(t)|\leq \int_\Omega \rho(t, x) |u(t, x)-\bar u(t)|\,dx\leq \bar\rho \,\|\partial_x u\|_{L^2}.
\end{equation}
Hence,
\begin{align}\label{poincare}
  \|\sqrt{\rho} u\|_{L^2}\leq&  \|\sqrt{\rho} (u-\bar u)\|_{L^2} +   |\bar u|\,\|\sqrt{\rho}  \|_{L^2}\notag\\
  \leq&  \sqrt{\bar \rho}  \|\partial_x u\|_{L^2}+ \bar\rho \,\|\partial_x u\|_{L^2}\notag\\
    \leq&   2\frac{ \bar{\rho}+1}{\mu_*}\|\sqrt{\mu(\rho)}\partial_x u\|_{L^2}.
\end{align}
Now, taking \eqref{poincare} into \eqref{es-500}  gives that
\begin{align}\label{a-2}
 \frac{d}{dt}\int_{\Omega}\partial_x^{-1}(\rho u)a\,dx + \frac{C}{2}\int_{\Omega}a^2\,dx  \leq  \left(8\frac{ (\bar{\rho}+1)^3}{\mu_*^2}+ \frac{\mu^*}{2C}\right)\|\sqrt{\mu(\rho)}\partial_x u\|_{L^2}^2.
\end{align}

It  can be seen from the fact that $\|e\|_{L^1}$ is equivalent to $\|a\|_{L^2}$  and \eqref{recal-1}, \eqref{poincare} and \eqref{a-2} that
for  small enough $\eta>0,$ it holds that
\begin{align}\label{basic-decay}
  \frac{d}{dt}\int_\Omega \left(e(\rho)+\rho u^2 + \eta\partial_x^{-1}(\rho u)\,a\right)dx + C\eta\int_{\Omega}(e(\rho)+ \rho u^2)dx\leq 0.
\end{align}
Noticing that by H\"{o}lder's inequality,
\begin{align*}
 \lt| \int_\Omega \partial_x^{-1}(\rho u)a\, dx\rt|\leq  \sqrt{\bar{\rho}}\,\|\sqrt{\rho} u\|_{L^2}\,\|a\|_{L^2}.
\end{align*}
This  finally implies that the above energy functional is  equivalent to $\|e\|_{L^1}+ \|\rho u\|_{L^2}^2,$ and hence the estimate
\begin{equation}\label{decay11-est}
  \|e\|_{L^1}+\|\sqrt{\rho}u\|_{L^2}\leq Ce^{-\alpha t},\quad \forall t\in\mathbb{R}_+.
\end{equation}
To show decay estimate of the velocity,  we notice from \eqref{decay11-est} that
\begin{align*}
\int_\Omega u^2(t, x)\,dx\leq& \int_\Omega \rho u^2(t, x)\,dx+ \int_\Omega |(1-\rho)| u^2(t, x)\,dx\\
\leq&  \|\sqrt{\rho} u(t)\|_{L^2}^2 + \|\rho-1\|_{L^2}\|u(t, \cdot)\|_{L^4}^2\\
\leq& C e^{-\alpha t}
\end{align*}
for any $t\geq t_0>0,$ where in the last inequality we used Lemma \ref{lemma3} and \eqref{decay11-est}.
This completes the proof of Proposition \ref{decay1}.
\end{proof}

\bigbreak

We can further derive an exponential decay estimates for $\|\pl_x u\|_{L^2}$ by taking use of Proposition \ref{decay1}
and under the assumption that viscosity coefficient $\mu$ \textit{is a constant}.      This is based on the following lemma.
\begin{Lemma}\label{ux-decay}
Assume that the viscosity coefficient $\mu$ is constant. Let the pressure $p$ satisfies \eqref{pcon} and $p(0)=0.$  Consider solutions to system \eqref{equ} subject to initial conditions \eqref{initialdata-Cond1} and \eqref{initialdata-Cond2}. Under the additional regularity assumption that $\partial_x u_0\in L^2(\Omega)$, there exists a positive constant $A_1$ such that for all $(t,x)\in \mathbb{R}_+\times \Omega$ the following estimate holds:
\begin{align}
    &\frac{d}{dt}\int_\Omega\left[A_1\rho|u|^4+\frac{\mu}{2} |\pl_xu|^2-(p-p(1))\pl_xu+\frac{1}{2\mu}\left(p^2-p^2(1)-2p(1)p'(1) a\right)\right]\, dx \notag\\
      &\hspace{1cm}+ \frac{1}{2}\int_\Omega \rho |\pl_tu|^2dx+\int_\Omega |u|^2|\pl_x u|^2 dx
  \leq  C(\overline{\rho})  \int_\Omega a^2dx +C(\overline{\rho}) \int_\Omega |\pl_x u|^2 dx .\label{ux-decay-est}
\end{align}
\end{Lemma}
\begin{proof}
Multiply \eqref{equ-b} by $4u^3$, and integrate with respect to space variable over $\Omega$ to obtain
\begin{align*}
\frac{d}{dt}\int_\Omega \rho |u|^4 dx + 12 \int_\Omega \mu  |u|^2|\pl_x u|^2 dx&= 12 \int_I p\pl_xu|u|^2dx\\
&\leq 6\int_\Omega \mu |u|^2|\pl_x u|^2 dx+C \|p^2(\rho)/\rho\|_{L^\infty}\frac{\bar{\rho}^2+\bar \rho}{\mu}\int_\Omega |\pl_xu|^2dx,
\end{align*}
where we have used $p(0)=0,p\in C^1$ and \eqref{poincare}. Then, we have
\begin{align}\label{u-4}
\frac{d}{dt}\int_\Omega \rho |u|^4 dx + 6 \int_\Omega \mu |u|^2|\pl_x u|^2 dx
\leq C(\|p^2(\rho)/\rho\|_{L^\infty} )\frac{\bar \rho^2+\bar\rho}{\mu}\int_\Omega |\pl_xu|^2dx.
\end{align}

Next, multiplying \eqref{equ-b} by $\pl_tu$, and integrating the resultant with respect to space variable over $\Omega$ gives
\begin{align}\label{u-t}
  \frac{1}{2}\frac{d}{dt}\int_\Omega \mu |\pl_xu|^2dx+ \int_\Omega \rho |\pl_tu|^2dx
=-\int_\Omega \rho u\pl_x u\pl_tudx+\int_\Omega p\pl_{xt}u dx.
\end{align}
The first term on the right-hand side of the above identity can be estimated as follows:
\begin{equation*}
  \int_\Omega \rho u\pl_xu\pl_tudx\leq \frac14\int_\Omega \rho |\pl_tu|^2dx+4\overline{\rho}\int_\Omega |u|^2|\pl_xu|^2dx.
\end{equation*}
For the second term, we have
\begin{align}\label{second}
  \int_\Omega p\pl_{xt}u=\frac{d}{dt}\int_\Omega (p-1)\pl_xu dx-\int_\Omega \pl_tp\pl_xudx.
\end{align}
Recall that the pressure $p$ satisfies
\begin{align*}
\pl_t p+\pl_x (pu)+h \pl_x u=0,
\end{align*}
where $h=\rho p'-p$. We can decompose the last term of \eqref{second} as
\begin{align*}
  \int_\Omega \pl_tp\pl_xudx&=\int_\Omega \frac{1}{\mu} \pl_tpp dx+\int_\Omega \frac{1}{\mu}\pl_tp Gdx\\
  &=\frac{1}{2\mu}\frac{d}{dt}\int_\Omega \left(p^2-p^2(1)-2p(1)p'(1) a\right)dx+\frac{1}{\mu}\int_\Omega pu\pl_xGdx-\frac{1}{\mu}\int_\Omega h\pl_xuGdx.
\end{align*}
Notice that $\pl_xG=\rho(\pl_tu+u\pl_xu)$, we apply Young's inequality and \eqref{poincare} to get
\begin{align*}
  \frac{1}{\mu}\int_\Omega pu\pl_xGdx&= \frac{1}{\mu} \int_\Omega pu\rho(\pl_tu+u\pl_xu)dx\\
  &\leq \frac{1}{4}\int_\Omega \rho|\pl_tu|^2dx+\frac{5}{\mu^2}p^2(\bar \rho)(\bar \rho^2+\bar \rho)\int_\Omega |\pl_xu|^2dx+\overline{\rho}\int_\Omega|u|^2|\pl_xu|^2dx.
\end{align*}
Using $u|_{\pl \Omega}=0,$ we derive that
\begin{align*}
  \int_\Omega h \pl_xu Gdx&\leq\|h\|_{L^\infty}\int_\Omega |\pl_xu|^2dx+\left|\int_ \Omega(p h-p(1)h(1))\pl_xudx\right|\\
  &\leq \|h\|_{L^\infty}\int_\Omega |\pl_xu|^2dx+C(\bar\rho) \int_\Omega a^2dx+\int_\Omega |\pl_xu|^2dx
\end{align*}
We combine the above estimates with \eqref{u-t} to obtain
\begin{align}\notag
 \frac{d}{dt}&\int_\Omega\left[\frac{\mu}{2} |\pl_xu|^2-(p-1)\pl_xu+\frac{1}{2\mu}\left(p^2-p^2(1)-2p(1)p'(1) a\right)\right]dx+ \frac{1}{2}\int_\Omega \rho |\pl_tu|^2dx\notag\\
  &\leq 5\overline{\rho}\int_\Omega|u|^2|\pl_xu|^2dx+C(\bar\rho) \int_\Omega a^2dx\notag\\
 & +\frac{1}{\mu}\lt( C(\|p^2(\rho)/\rho\|_{L^\infty} ){(\bar \rho^2+\bar\rho)}+\|h\|_{L^\infty}+\frac{5}{\mu}p^2(\bar \rho)(\bar\rho^2+\bar\rho)+1\rt)\int_\Omega|\pl_xu|^2dx.\label{u-t-1}
\end{align}

Finally, we calculate $A_1\times \eqref{u-4}+\eqref{u-t-1}$ and choose $A_1$ such that  $6\mu A_1=5\overline{\rho}+1$. Hence we obtain \eqref{ux-decay-est} with the help of the boundedness of $\|h\|_{L^\infty}$ and $\|p^2/\rho\|_{L^\infty}.$

\end{proof}

Now, multiplying \eqref{basic-decay} by a large enough  constant $A_2$ then adding to \eqref{ux-decay-est}, we can achieve
\begin{equation*}
  \frac{d}{dt}\mathfrak{E}_1(t)+\int_\Omega\left(a^2+|\pl_x u|^2+ \rho |\pl_tu|^2+ |u|^2|\pl_x u|^2 \right)dx\leq 0,
\end{equation*}
with
\begin{align*}
  \mathfrak{E}_1(t):=\int_\Omega &\left[A_2\left(e(\rho)+\rho u^2\right)+A_1\rho|u|^4+\frac{\mu}{2} |\pl_xu|^2\right.\\
  &\left.-(p-p(1))\pl_xu+\frac{1}{2\mu}\left(p^2-p^2(1)-2p'(1)p(1)a\right)\right]dx.
\end{align*}
Notice that
\begin{equation*}
  \left|\int_\Omega(p-1)\pl_xudx\right|\leq \frac{p'(\xi)^2}{\mu}\int_\Omega a^2dx+\frac{\mu}{4}\int_\Omega|\pl_xu|^2dx,
\end{equation*}
for some $\xi\in[0,\bar\rho]$. By choosing $A_2$ large enough, we have
\begin{equation*}
  \mathfrak{E}_1\thicksim \int_\Omega\left(e(\rho)+\rho u^2+\rho|u|^4+|\pl_xu|^2\right)dx.
\end{equation*}
Observing that
\begin{equation*}
  \int_\Omega\rho u^4dx\leq 4\overline{\rho}^2\int_\Omega|u|^2|\pl_xu|^2dx+\left(\int_\Omega\rho u^2dx\right)^2,
\end{equation*}
which combines with the exponential decay estimates for $\|\sqrt{\rho}u\|_{L^2}$ in \eqref{decay1-est} yields the following proposition.
\begin{Proposition}\label{decay2}
Under the assumptions of Lemma \ref{ux-decay}, the following estimate holds:
\begin{equation}\label{decay2-est}
  \|\pl_xu\|_{L^2} \leq Ce^{-\alpha t},\quad \forall t\in\mathbb{R}_+,\, x\in\Omega,
\end{equation}
for some constants $C,\alpha$ depending only on $M_0, E_0, \mu, \rho^*$.
\end{Proposition}

\section{Proof of Corollary \ref{cor1}}
For expository reasons, here we will only prove the vacuum bubble dynamics, i.e. inward expanding  flow. The analysis for the density patch dynamics (outward expanding  flow) can be done similarly.    Use $x=a(t)$ and $x=b(t)$ to denote the particle paths start from $a_0$ and $b_0$,
respectively,  which are the interfaces between the gas and the vacuum, i.e.
\begin{equation}\label{ODE}
\left\{\begin{aligned}
&\rho=0,\quad x\in[a(t),b(t)],\quad t\geq0,\\
&\frac{\mathrm{d} a(t)}{\mathrm{d}t}=u(t,a(t)),\quad a(0):=a_0,\quad t\geq0,  \\
&\frac{\mathrm{d} b(t)}{\mathrm{d}t}=u(t,b(t)),\quad b(0):=b_0, \quad t\geq0, \\
&\rho>0,\quad x\in(0,a(t))\cup(b(t),1),\quad t\geq0.
\end{aligned}\right.
\end{equation}
Note that the above scenario can be easily seen  from the Lagrangian trajectory as in Section \ref{S4}.
It is important to note  that each individual portion of the gas itself obeys mass conservation laws:
\begin{equation*}
\int_0^{a(t)}\rho(t,x) dx=\int_0^{a_0}\rho_0 dx:=M_L.
\end{equation*}
Indeed,
\begin{align*}
  \frac{\mathrm{d}}{\mathrm{d}t}\int_0^{a(t)}\rho(t,x) dx&=\frac{\mathrm{d} a(t)}{\mathrm{d}t}\rho(t,a(t))+\int_0^{a(t)}\pl_t\rho(t,x) dx\\
  &=u(t,a(t))\rho(t,a(t))-\int_0^{a(t)}\pl_x(\rho u)dx=0.
\end{align*}

To prove \eqref{cor-est1}, we first focus on the behavior of $a(t)$. Due to conservation laws of mass, Proposition \ref{decay1} and H\"{o}lder's inequality, we write that
\begin{align*}
  \left|a(t)-M_L\right|&=\left|a(t)-\int_0^{a(t)}\rho(t,x)dx\right|\\
  &\leq \int_\Omega|1-\rho(t,x)|dx\leq \|\rho-1\|_{L^2}\leq Ce^{-\alpha t}.
\end{align*}

As for $b(t)$, notice that $\|\rho(t, \cdot)\|_{L^1(\Omega)}=1$ and thus
\begin{equation*}
  \left|M_L-b(t)\right|=\left|(1-M_L)-(1-b(t))\right|=\left|\int_{b(t)}^{1(t)}(\rho(t,x)-1)dx\right|\leq Ce^{-\alpha t}.
\end{equation*}
Hence,  we have
\begin{equation*}
  |a(t)-M_L| +|b(t)-M_L|\leq Ce^{-\alpha t}.
\end{equation*}\qed
\bigbreak

\textbf{Acknowledgements}
Part of the work was completed during the authors'  visit to the Tianyuan Mathematical Research Center.
The authors are grateful for the support and warm hospitality  provided by the Tianyuan Mathematical Research Center.     Tan is partially supported by a  Direct Grant (No. 4053715) of CUHK and the Hong Kong RGC under grant CUHK14302525.
Wang would like to thank Professor Chaojiang Xu for his comments and suggestions on this manuscript,  during Wang's visit to School of Mathematics, NUAA.
Wang's research was partially supported by Tianyuan Mathematics Project of NSFC under grant 12426630; and Scientific Research Fund of Zhejiang Provincial Education Department under project Y202454255.   Zhang was supported by National Natural Science Foundation of China (Grants No. 12101472 and No. 12471210) and the Fundamental Research Funds for the Central Universities.
\medskip

\textbf{Data Availability Statement} Data sharing not applicable to this article as no
datasets were generated or analysed during the current study.

\medskip

\textbf{Conflict of interest}  The authors have no conflict of interest to declare.

\normalem

\begin{thebibliography}{10}

\bibitem{MR3862947}
{\sc D.~Bresch and P.-E. Jabin}, {\em Global existence of weak solutions for
  compressible {N}avier-{S}tokes equations: thermodynamically unstable pressure
  and anisotropic viscous stress tensor}, Ann. of Math. (2), 188 (2018),
  pp.~577--684.

\bibitem{MR1789926}
{\sc G.-Q. Chen, D.~Hoff, and K.~Trivisa}, {\em Global solutions of the
  compressible {N}avier-{S}tokes equations with large discontinuous initial
  data}, Comm. Partial Differential Equations, 25 (2000), pp.~2233--2257.

\bibitem{MR4659290}
{\sc K.~Chen, L.~K. Ha, R.~Hu, and Q.-H. Nguyen}, {\em Global well-posedness of
  the 1d compressible {N}avier-{S}tokes system with rough data}, J. Math. Pures
  Appl. (9), 179 (2023), pp.~425--453.

\bibitem{MR4695791}
{\sc K.~Chen, R.~Hu, and Q.-H. Nguyen}, {\em Local well-posedness of the 1d
  compressible {N}avier-{S}tokes system with rough data}, Calc. Var. Partial
  Differential Equations, 63 (2024), pp.~Paper No. 42, 45.

\bibitem{MR2038120}
{\sc Y.~Cho, H.~J. Choe, and H.~Kim}, {\em Unique solvability of the initial
  boundary value problems for compressible viscous fluids}, J. Math. Pures
  Appl. (9), 83 (2004), pp.~243--275.

\bibitem{MR2223483}
{\sc Y.~Cho and H.~Kim}, {\em On classical solutions of the compressible
  {N}avier-{S}tokes equations with nonnegative initial densities}, Manuscripta
  Math., 120 (2006), pp.~91--129.

\bibitem{MR1779621}
{\sc R.~Danchin}, {\em Global existence in critical spaces for compressible
  {N}avier-{S}tokes equations}, Invent. Math., 141 (2000), pp.~579--614.

\bibitem{MR4047647}
{\sc R.~Danchin, F.~Fanelli, and M.~Paicu}, {\em A well-posedness result for
  viscous compressible fluids with only bounded density}, Anal. PDE, 13 (2020),
  pp.~275--316.

\bibitem{MR4642822}
{\sc R.~Danchin and P.~B.~a. Mucha}, {\em Compressible {N}avier-{S}tokes
  equations with ripped density}, Comm. Pure Appl. Math., 76 (2023),
  pp.~3437--3492.

\bibitem{MR4718745}
{\sc R.~Danchin and S.~Wang}, {\em Exponential decay for inhomogeneous viscous
  flows on the torus}, Z. Angew. Math. Phys., 75 (2024), pp.~Paper No. 62, 33.

\bibitem{MR1022305}
{\sc R.~J. DiPerna and P.-L. Lions}, {\em Ordinary differential equations,
  transport theory and {S}obolev spaces}, Invent. Math., 98 (1989),
  pp.~511--547.

\bibitem{MR2237707}
{\sc D.~Fang and T.~Zhang}, {\em Compressible {N}avier-{S}tokes equations with
  vacuum state in the case of general pressure law}, Math. Methods Appl. Sci.,
  29 (2006), pp.~1081--1106.

\bibitem{MR2499296}
{\sc E.~Feireisl and A.~Novotn\'{y}}, {\em Singular limits in thermodynamics of
  viscous fluids}, Advances in Mathematical Fluid Mechanics, Birkh\"{a}user
  Verlag, Basel, 2009.

\bibitem{MR1867887}
{\sc E.~Feireisl, A.~Novotn\'{y}, and H.~Petzeltov\'{a}}, {\em On the existence
  of globally defined weak solutions to the {N}avier-{S}tokes equations}, J.
  Math. Fluid Mech., 3 (2001), pp.~358--392.

\bibitem{MR2644149}
{\sc Z.~Guo and C.~Zhu}, {\em Global weak solutions and asymptotic behavior to
  1{D} compressible {N}avier-{S}tokes equations with density-dependent
  viscosity and vacuum}, J. Differential Equations, 248 (2010), pp.~2768--2799.

\bibitem{MR896014}
{\sc D.~Hoff}, {\em Global existence for {$1$}{D}, compressible, isentropic
  {N}avier-{S}tokes equations with large initial data}, Trans. Amer. Math.
  Soc., 303 (1987), pp.~169--181.

\bibitem{MR1088275}
\leavevmode\vrule height 2pt depth -1.6pt width 23pt, {\em Discontinuous
  solutions of the {N}avier-{S}tokes equations for compressible flow}, Arch.
  Rational Mech. Anal., 114 (1991), pp.~15--46.

\bibitem{MR1142276}
\leavevmode\vrule height 2pt depth -1.6pt width 23pt, {\em Global
  well-posedness of the {C}auchy problem for the {N}avier-{S}tokes equations of
  nonisentropic flow with discontinuous initial data}, J. Differential
  Equations, 95 (1992), pp.~33--74.

\bibitem{HOff98}
\leavevmode\vrule height 2pt depth -1.6pt width 23pt, {\em Global solutions of
  the equations of one-dimensional, compressible flow with large data and
  forces, and with differing end states}, Z. angew. Math. Phys., 49 (1998),
  pp.~774--785.

\bibitem{MR1117422}
{\sc D.~Hoff and D.~Serre}, {\em The failure of continuous dependence on
  initial data for the {N}avier-{S}tokes equations of compressible flow}, SIAM
  J. Appl. Math., 51 (1991), pp.~887--898.

\bibitem{MR3926039}
{\sc G.~Hong and C.~Zhu}, {\em Optimal decay rates on compressible
  {N}avier-{S}tokes equations with degenerate viscosity and vacuum}, J. Math.
  Pures Appl. (9), 124 (2019), pp.~1--29.

\bibitem{MR3505779}
{\sc X.~Huang and J.~Li}, {\em Existence and blowup behavior of global strong
  solutions to the two-dimensional barotrpic compressible {N}avier-{S}tokes
  system with vacuum and large initial data}, J. Math. Pures Appl. (9), 106
  (2016), pp.~123--154.

\bibitem{MR2877344}
{\sc X.~Huang, J.~Li, and Z.~Xin}, {\em Global well-posedness of classical
  solutions with large oscillations and vacuum to the three-dimensional
  isentropic compressible {N}avier-{S}tokes equations}, Comm. Pure Appl. Math.,
  65 (2012), pp.~549--585.

\bibitem{MR2254008}
{\sc S.~Jiang, Z.~Xin, and P.~Zhang}, {\em Global weak solutions to 1{D}
  compressible isentropic {N}avier-{S}tokes equations with density-dependent
  viscosity}, Methods Appl. Anal., 12 (2005), pp.~239--251.

\bibitem{MR2005201}
{\sc S.~Jiang and P.~Zhang}, {\em Axisymmetric solutions of the 3{D}
  {N}avier-{S}tokes equations for compressible isentropic fluids}, J. Math.
  Pures Appl. (9), 82 (2003), pp.~949--973.

\bibitem{MR3247365}
{\sc Q.~Jiu, Y.~Wang, and Z.~Xin}, {\em Global well-posedness of 2{D}
  compressible {N}avier-{S}tokes equations with large data and vacuum}, J.
  Math. Fluid Mech., 16 (2014), pp.~483--521.

\bibitem{MR227619}
{\sc J.~I. Kanel$^\prime$}, {\em A model system of equations for the
  one-dimensional motion of a gas}, Differencial$^\prime$ nye Uravnenija, 4
  (1968), pp.~721--734.

\bibitem{MR468593}
{\sc A.~V. Kazhikhov and V.~V. Shelukhin}, {\em Unique global solution with
  respect to time of initial-boundary value problems for one-dimensional
  equations of a viscous gas}, Prikl. Mat. Meh., 41 (1977), pp.~282--291.

\bibitem{MR2410901}
{\sc H.-L. Li, J.~Li, and Z.~Xin}, {\em Vanishing of vacuum states and blow-up
  phenomena of the compressible {N}avier-{S}tokes equations}, Comm. Math.
  Phys., 281 (2008), pp.~401--444.

\bibitem{MR3923730}
{\sc J.~Li and Z.~Xin}, {\em Global well-posedness and large time asymptotic
  behavior of classical solutions to the compressible {N}avier-{S}tokes
  equations with vacuum}, Ann. PDE, 5 (2019), pp.~Paper No. 7, 37.

\bibitem{MR1637634}
{\sc P.-L. Lions}, {\em Mathematical topics in fluid mechanics. {V}ol. 2},
  vol.~10 of Oxford Lecture Series in Mathematics and its Applications, The
  Clarendon Press, Oxford University Press, New York, 1998.
\newblock Compressible models, Oxford Science Publications.

\bibitem{MR1485360}
{\sc T.-P. Liu, Z.~Xin, and T.~Yang}, {\em Vacuum states for compressible
  flow}, Discrete Contin. Dynam. Systems, 4 (1998), pp.~1--32.

\bibitem{MR4373171}
{\sc T.-P. Liu and S.-H. Yu}, {\em Navier-{S}tokes equations in gas dynamics:
  {G}reen's function, singularity, and well-posedness}, Comm. Pure Appl. Math.,
  75 (2022), pp.~223--348.

\bibitem{MR1766564}
{\sc T.~Luo, Z.~Xin, and T.~Yang}, {\em Interface behavior of compressible
  {N}avier-{S}tokes equations with vacuum}, SIAM J. Math. Anal., 31 (2000),
  pp.~1175--1191.

\bibitem{MR882389}
{\sc T.~Makino}, {\em On a local existence theorem for the evolution equation
  of gaseous stars}, in Patterns and waves, vol.~18 of Stud. Math. Appl.,
  North-Holland, Amsterdam, 1986, pp.~459--479.

\bibitem{MR564670}
{\sc A.~Matsumura and T.~Nishida}, {\em The initial value problem for the
  equations of motion of viscous and heat-conductive gases}, J. Math. Kyoto
  Univ., 20 (1980), pp.~67--104.

\bibitem{MR149094}
{\sc J.~Nash}, {\em Le probl\`eme de {C}auchy pour les \'{e}quations
  diff\'{e}rentielles d'un fluide g\'{e}n\'{e}ral}, Bull. Soc. Math. France, 90
  (1962), pp.~487--497.

\bibitem{MR981519}
{\sc M.~Okada}, {\em Free boundary value problems for the equation of
  one-dimensional motion of viscous gas}, Japan J. Appl. Math., 6 (1989),
  pp.~161--177.

\bibitem{MR1227730}
{\sc M.~Okada and T.~Makino}, {\em Free boundary problem for the equation of
  spherically symmetric motion of viscous gas}, Japan J. Indust. Appl. Math.,
  10 (1993), pp.~219--235.

\bibitem{MR1980822}
{\sc M.~Okada, v.~Matu\v{s}{u}-Ne\v{c}asov\'{a}, and T.~Makino}, {\em Free
  boundary problem for the equation of one-dimensional motion of compressible
  gas with density-dependent viscosity}, Ann. Univ. Ferrara Sez. VII (N.S.), 48
  (2002), pp.~1--20.

\bibitem{MR2373221}
{\sc X.~Qin, Z.-A. Yao, and H.~Zhao}, {\em One dimensional compressible
  {N}avier-{S}tokes equations with density-dependent viscosity and free
  boundaries}, Commun. Pure Appl. Anal., 7 (2008), pp.~373--381.

\bibitem{MR1217657}
{\sc R.~Salvi and I.~Stra\v{s}kraba}, {\em Global existence for viscous
  compressible fluids and their behavior as {$t\to\infty$}}, J. Fac. Sci. Univ.
  Tokyo Sect. IA Math., 40 (1993), pp.~17--51.

\bibitem{MR2927617}
{\sc A.~Suen and D.~Hoff}, {\em Global low-energy weak solutions of the
  equations of three-dimensional compressible magnetohydrodynamics}, Arch.
  Ration. Mech. Anal., 205 (2012), pp.~27--58.

\bibitem{MR1375428}
{\sc V.~A. Va\u{\i}gant and A.~V. Kazhikhov}, {\em On the existence of global
  solutions of two-dimensional {N}avier-{S}tokes equations of a compressible
  viscous fluid}, Sibirsk. Mat. Zh., 36 (1995), pp.~1283--1316, ii.

\bibitem{MR4444077}
{\sc H.~Wang, S.-H. Yu, and X.~Zhang}, {\em Global well-posedness of
  compressible {N}avier-{S}tokes equation with {$BV\cap L^1$} initial data},
  Arch. Ration. Mech. Anal., 245 (2022), pp.~375--477.

\bibitem{MR4579720}
{\sc H.~Wang and X.~Zhang}, {\em Propagation of rough initial data for
  {N}avier-{S}tokes equation}, SIAM J. Math. Anal., 55 (2023), pp.~966--1006.

\bibitem{MR3597161}
{\sc H.~Wen and C.~Zhu}, {\em Global solutions to the three-dimensional full
  compressible {N}avier-{S}tokes equations with vacuum at infinity in some
  classes of large data}, SIAM J. Math. Anal., 49 (2017), pp.~162--221.

\bibitem{MR1843291}
{\sc T.~Yang, Z.-a. Yao, and C.~Zhu}, {\em Compressible {N}avier-{S}tokes
  equations with density-dependent viscosity and vacuum}, Comm. Partial
  Differential Equations, 26 (2001), pp.~965--981.

\bibitem{MR1929151}
{\sc T.~Yang and H.~Zhao}, {\em A vacuum problem for the one-dimensional
  compressible {N}avier-{S}tokes equations with density-dependent viscosity},
  J. Differential Equations, 184 (2002), pp.~163--184.

\bibitem{MR1936794}
{\sc T.~Yang and C.~Zhu}, {\em Compressible {N}avier-{S}tokes equations with
  degenerate viscosity coefficient and vacuum}, Comm. Math. Phys., 230 (2002),
  pp.~329--363.

\bibitem{MR2563807}
{\sc C.~Zhu}, {\em Asymptotic behavior of compressible {N}avier-{S}tokes
  equations with density-dependent viscosity and vacuum}, Comm. Math. Phys.,
  293 (2010), pp.~279--299.

\end{thebibliography}

\end{document}